\documentclass[11pt]{amsart} 

\usepackage[dvips]{graphics} 
\usepackage{color,amssymb,amsbsy,amsmath,amsfonts,amssymb,
amscd,epsfig,times,graphics}
\newtheorem{theo}{Theorem}[section] 
\newtheorem{prop}[theo]{Proposition}
\newtheorem{ex}[theo]{Example}
\newtheorem{lem}[theo]{Lemma} 
\newtheorem{cor}[theo]{Corollary}
\newtheorem{rem}[theo]{Remark}
\theoremstyle{definition}
\newtheorem{defi}[theo]{Definition}

\numberwithin{equation}{section}

\newcommand{\C}{\mathbb C} 
\newcommand{\R}{\mathbb R} 
 
\newcommand{\B}{\mathbb B}

\newcommand{\partx}{\partial/\partial x}
\begin{document} 
\begin{abstract}
Let $D=\{\rho<0\}$ be a smooth relatively compact domain  in a four dimensional almost  
complex manifold $(M,J)$, where $\rho$ is
a $J$-plurisubharmonic function on a neighborhood of $\overline{D}$ and strictly 
$J$-plurisubharmonic on a neighborhood of $\partial D$. We give sharp estimates of the Kobayashi metric. 
Our approach  is based on an asymptotic quantitative  description 
of both the domain $D$ and the almost complex structure $J$ near a boundary point. Following Z.M.Balogh and M.Bonk \cite{ba-bo}, 
these  sharp estimates provide the Gromov hyperbolicity of the  domain $D$.
\end{abstract}
\title[Sharp estimates of the Kobayashi metric and Gromov~hyperbolicity]
{ Sharp estimates of the Kobayashi metric and Gromov~hyperbolicity} 
\author{Florian Bertrand}
\address{LATP, C.M.I, 39 rue Joliot-Curie 13453 Marseille cedex 13, FRANCE }
\email{bertrand@cmi.univ-mrs.fr}
\maketitle 
\section*{Introduction} 

One can define different notions of hyperbolicity  on a given manifold, based on geometric structures, 
and it seems natural to try to connect them. For instance, the links between the symplectic hyperbolicity and the Kobayashi 
hyperbolicity were studied by A.-L.Biolley \cite{bio}. In the article \cite{ba-bo}, Z.M.Balogh and M.Bonk established
deep connections between the Kobayashi hyperbolicity and the Gromov hyperbolocity, based on sharp asymptotic estimates of the 
Kobayashi metric. Since the Gromov hyperbolicity may be defined on any geodesic space, it is natural to 
understand its links with the Kobayashi hyperbolicity in the most general manifolds on which the Kobayashi metric can be defined, 
namely the almost complex manifolds.  As emphasized by \cite{ba-bo}, it is necessary 
to study precisely the Kobayashi metric. Since there is no exact expression of this pseudometric, 
except for particular domains where geodesics can be determined explicitely, we are interested in the 
boundary behaviour of the Kobayashi metric and  in its asymptotic geodesics.
One can note that boundary estimates of this  invariant pseudometric, whose existence is 
directly issued from the existence of pseudoholomorphic discs proved by A.Nijenhuis-W.Woolf  \cite{ni-wo}, is also 
a fundamental tool for the study of the extension of diffeomorphisms and for the classification of manifolds. 
%and it  gives important 
%informations on the geometry and the dynamics of the manifolds. 

The first results in this direction are due to I.Graham \cite{gra}, who gave boundary  estimates of the  
Kobayashi metric near a strictly pseudoconvex boundary point, providing the (local) 
complete hyperbolicity near such a point. Considering a $L^2$-theory approach, D.Catlin \cite{ca} obtained similar estimates on 
pseudoconvex domains of finite type in $\C^2$. A crucial progress in the strictly pseudoconvex case is due to 
D.Ma \cite{ma}, who gave an optimal asymptotic description of this metric. His approach is 
based on a localization principle given by F.Forstneric and J.-P.Rosay \cite{fo-ro} using some purely complex 
analysis arguments as peak holomorphic functions. The estimates proved by D.Ma were used in \cite{ba-bo} to prove the 
Gromov hyperbolicity of relatively compact strictly pseudoconvex domains. The aim of this paper is to obtain 
sharp estimates of the Kobayashi metric on strictly pseudoconvex domains in four almost complex manifolds: 
\vskip 0,3cm
\noindent{\bf Theorem A}. 
{\it Let $D$ be a relatively compact strictly $J$-pseudoconvex smooth domain 
in a four dimensional almost complex manifold $(M,J)$.  
Then for every $\varepsilon>0$, there exists $0<\varepsilon_0<\varepsilon$ and positive constants $C$ and $s$
such that for every $p \in D\cap N_{\varepsilon_0}(\partial D)$ and every $v=v_n+v_t \in T_{p}M$ we have 
\begin{eqnarray}\label{3est2a}
e^{-C\delta(p)^s}\left(\frac{|v_n|^2}{4\delta(p)^2}+
\frac{\mathcal{L}_J\rho(\pi(p),v_t)}{2\delta(p)}\right)^{\frac{1}{2}} 
&\leq &K_{(D,J)}(p,v) \nonumber \\
&&\nonumber\\
&&\hspace{0.75cm}\leq e^{C\delta(p)^s}\left(\frac{|v_n|^2}{4\delta(p)^2}+
\frac{\mathcal{L}_J\rho(\pi(p),v_t)}{2\delta(p)}\right)^{\frac{1}{2}}.
\end{eqnarray}}
\vskip 0,3cm
In the above theorem, $\delta(p):={\rm dist}(p,\partial D)$, where $\rm{dist}$ is taken with respect to a Riemannian metric. 
For $p$ sufficiently close to the boundary the point $\pi(p)$ denotes the unique boundary point such that $\delta(p)=\|p-\pi(p)\|$. 
Moreover $N_{\varepsilon_0}(\partial D):=\{q \in M, \delta(q)<\varepsilon_0\}$. 
We point out that the splitting $v=v_n+v_t \in T_{p}M$ in tangent and normal components in (\ref{3est2a}) 
is understood to be taken at $\pi(p)$.

%Our proof is inspired by result by D.Ma \cite{ma}. However the proof he gives is based on some purely 
%complex analysis argument as the local existence of peak holomorphic functions. Since such functions do not exist generically
%in almost complex manifolds, we consider a quantitative approach by introducing a well chosen family of polydiscs which allows 
%to work locally and to control the almost complex structure induced by some scaling method. 

%Notice that this 
%also gives a different way to obtain sharp estimates of the Kobayashi pseudometric given by D.Ma \cite{ma} in complex manifolds 
%without using any complex analysis tools.

%In the complex Euclidean space, Z.M.Balogh and M.Bonk \cite{ba-bo} 
%proved the Gromov hyperbolicity of strictly pseudoconvex domains. Their proof is based on sharp estimates of 
%the Kobayashi pseudometric obtained by D.Ma \cite{ma} similar to the ones 
%provided by (\ref{3est2a}), and on some sub-Riemannian geometry.
\vskip 0,5cm
As a corollary of Theorem A, we obtain:  
\vskip 0,3cm
\noindent{\bf Theorem B}.{\it 
\begin{enumerate} 
\item Let $D$ be a relatively compact strictly $J$-pseudoconvex smooth domain 
in an almost complex manifold $(M,J)$  of dimension four. Then the domain $D$ endowed with the Kobayashi integrated distance 
$d_{(D,J)}$ is a Gromov hyperbolic metric space.
\item Each point in a four dimensional almost complex manifold admits a 
basis of Gromov hyperbolic neighborhoods.
\end{enumerate}
}
\vskip 0,3cm

The paper is organized as follows. In Section 1, we give general facts about almost complex manifolds. In Section 2, we
show how to deduce Theorem B from Theorem A. Finally, Section 3 is devoted to the proof of our main result, namely 
Theorem A.

\section{Preliminaries}
We denote by $\Delta$ the unit disc of $\C$ and by $\Delta_{r}$ 
the disc of $\C$ centered at the origin of radius $r>0$.
\subsection{Almost complex manifolds and pseudoholomorphic discs}

An {\it almost complex structure} $J$ on a real smooth manifold $M$ is a $\left(1,1\right)$ tensor field
 which  satisfies $J^{2}=-Id$. We suppose that $J$ is smooth.
The pair $\left(M,J\right)$ is called an {\it almost complex manifold}. We denote by $J_{st}$
the standard integrable structure on $\C^{n}$ for every $n$.
A differentiable map $f:\left(M',J'\right) \longrightarrow \left(M,J\right)$ between two almost complex manifolds is said to be 
 {\it $\left(J',J\right)$-holomorphic}  if $J\left(f\left(p\right)\right)\circ d_{p}f=d_{p}f\circ J'\left(p\right),$ 
for every $p \in M'$. In case  $M'=\Delta \subset \C$, such a map is called a {\it pseudoholomorphic disc}.  
If $f:\left(M,J\right)\longrightarrow M'$ is a diffeomorphism, we define an almost complex structure, $f_{*}J$,  
on $M'$ as the
 direct image of $J$ by $f$:
$$f_{*}J\left(q\right):=d_{f^{-1}\left(q\right)}f\circ J\left(f^{-1}\left(q\right)\right)\circ d_{q}f^{-1},$$ 
 for every  $q \in M'$.

The following lemma (see \cite{ga-su}) states that locally any almost
 complex manifold can be seen as the unit ball of 
$\C^n$ endowed with a small smooth pertubation of the standard
 integrable structure $J_{st}$. 
\begin{lem}\label{ilemloc}
Let $\left(M,J\right)$ be an almost complex manifold, with $J$ of class
 $\mathcal{C}^{k}$, $k\geq 0$. 
Then for every point $p \in M$ and every $\lambda_0 > 0$ there exist a neighborhood $U$ of $p$ and
 a coordinate diffeomorphism $z: U \rightarrow \B$ centered a $p$ (ie $z(p) =0$) such that  the
direct image of $J$ satisfies $z_{*}J\left(0\right) = J_{st}$ and
$||z_*\left(J\right) - J_{st}||_{\mathcal{C}^k\left(\bar {\B}\right)}
 \leq \lambda_0$.
\end{lem}
This is simply done by considering a local chart $z: U \rightarrow \B$ centered a $p$ (ie $z(p) =0$), composing it 
with a linear diffeomorphism to insure $z_{*}J\left(0\right) = J_{st}$ and  dilating coordinates.

So let  $J$ be an almost complex structure defined in  a neighborhood $U$ of the origin in $\R^{2n}$, and such 
that $J$ is sufficiently closed to the standard structure in uniform norm on the closure $\overline{U}$ of $U$. 
The $J$-holomorphy equation for a pseudoholomorphic disc 
$u : \Delta \rightarrow U \subseteq \R^{2n}$ is given by  
\begin{equation}\label{eqholo}
\frac{\partial u}{\partial y}-J\left(u\right)\frac{\partial u}{\partial x}=0.
\end{equation}

According to \cite{ni-wo}, for every $p \in M$, there is a neighborhood $V$ of zero in $T_{p}M$, such that for every $v \in V$, 
there is a $J$-holomorphic disc $u$ satisfying $u\left(0\right)=p$ and $d_{0}u\left(\partx\right)=v$.

\subsection{Splitting of the tangent space}

Assume that $J$ is a diagonal almost complex structure defined in a neighborhood of the origin in $\R^4$ and such that 
$J(0)=J_{st}$. Consider a basis $(\omega_1,\omega_2)$ of $(1,0)$ differential forms
for the structure $J$ in a neighborhood of the origin. Since $J$ is
diagonal, we may choose 
$$\omega_j = dz^j - B_{j}(z)d\bar z^j, \mbox{ } j=1,2.$$
Denote by $(Y_1,Y_2)$ the corresponding dual basis
of $(1,0)$ vector fields. Then 
$$Y_j = \frac{\partial}{\partial z^j} -\beta_j(z)\frac{\partial}{\partial \overline{z^j}}, \mbox{ }j=1,2.$$ 
Moreover $B_j(0) = \beta_j(0) = 0$ for $j=1,2$. The basis $(Y_1(0),Y_2(0))$ simply coincides with the canonical (1,0)
basis of $\C^2$.
In particular $Y_1(0)$ is a basis vector of the complex tangent space
$T^J_0(\partial D)$  and $Y_2(0)$ is normal to $\partial D$.
Consider now for $t \geq 0$ the translation $\partial D -
t$ of the boundary of $D$ near the origin. Consider, in a neighborhood of the
origin, a $(1,0)$ vector field $X_1$ (for $J$) such that $X_1(0) = Y_1(0)$
and $X_1(z)$ generates the $J$-invariant tangent space $T^J_z(\partial D - t)$ at
every point $z \in \partial D - t$, $0 \leq t <<1$.
Setting $X_2 = Y_2$, we obtain a basis of vector fields
$(X_1,X_2)$ on $D$ (restricting $D$ if necessary). 
Any complex tangent vector $v \in T_z^{(1,0)}(D,J)$ at
point  $z \in D$ admits the unique
decomposition $v = v_t + v_n$ where $v_t = \alpha_1
X_1(z)$ is the tangent component and $v_n = \alpha_2 X_2(z)$ is the normal
component. Identifying $T_z^{(1,0)}(D,J)$ with $T_zD$ we may
consider the decomposition $v=v_t + v_n$ for each $v \in T_z(D)$.
Finally we consider this decomposition for points $z$ in a neighborhood of
the boundary.

\subsection{Levi geometry}

Let $\rho$ be a $\mathcal{C}^2$ real valued function on
a smooth almost complex manifold  $\left(M,J\right).$
We denote by $d^c_J\rho$ the differential 
form defined by 
\begin{equation}\label{ieqdc}
d^c_J\rho\left(v\right):=-d\rho\left(Jv\right),
\end{equation}
where  $v$ is a section of $TM$. 
The {\it Levi form} of $\rho$ at a point $p\in M$ and a vector 
$v \in T_pM$ is defined by
\begin{equation*}
\mathcal{L}_J\rho\left(p,v\right):=d\left(d^c_J\rho\right)(p)\left(v,J(p)v\right)=dd^c_J\rho(p)\left(v,J(p)v\right).
\end{equation*}
In case $(M,J)=(\C^n,J_{st})$, then $\mathcal{L}_{J_{st}}\rho$ is, up to a positive multiplicative constant, 
the usual standard Levi form:
\begin{equation*}
\displaystyle \mathcal{L}_{J_{st}}\rho(p,v)=4 \sum \frac{\partial^2\rho}{\partial z_j\partial \overline{z_k}}v_j\overline{v_k}.
\end{equation*}

\vspace{0,3cm}

We investigate now how close is the Levi form with respect to $J$ from the standard Levi form. 
For $p \in M$ and $v\in T_pM$, we easily get:
\begin{equation}\label{ieqlevi}
\mathcal{L}_J\rho\left(p,v\right)=\displaystyle \mathcal{L}_{J_{st}}\rho(p,v)+d(d^c_{J}-d^c_{J_{st}})\rho(p)(v,J(p)v)+
dd^c_{J_{st}}\rho(p)(v,J(p)-J_{st})v).
\end{equation}
In local coordinates $(t_1,t_2,\cdots,t_{2n})$ of $\R^{2n}$, (\ref{ieqlevi}) may be written as follows
\begin{eqnarray}\label{ieqlevi2}
\mathcal{L}_J\rho\left(p,v\right)&=&\displaystyle \mathcal{L}_{J_{st}}\rho(p,v)+{}^tv(A-{}^tA)J(p)v+
{ }^t(J(p)-J_{st})vDJ_{st}v+ \nonumber\\
& & { }^t(J(p)-J_{st})vD(J(p)-J_{st})v
\end{eqnarray}
where 
$$A:=\left(\sum_i\frac{\partial u}{\partial t_i}\frac{\partial
    J_{i,j}}{\partial t_k}\right)_{1\le j,k\le 2n}\quad \mathrm{and}\quad
D:=\left(\frac{\partial^2u}{\partial t_j\partial
    t_k}\right)_{1\le j,k\le 2n}.$$

\vspace{0,3cm}

 Let $f$ be a $(J',J)$-biholomorphism 
from $\left(M',J'\right)$ to $\left(M,J\right)$. Then for every $p\in M$ and every $v\in T_pM$:
$$\mathcal{L}_{J'}\rho\left(p,v\right)=\mathcal{L}_{J}\rho\circ
 f^{-1}\left(f\left(p\right),d_pf\left(v\right)\right).$$
This expresses the invariance of the Levi form under pseudobiholomorphisms.

The next proposition is useful in order to compute the Levi form (see \cite{iv-ro}).   
\begin{prop}\label{proplevi}\mbox{ }
Let $p\in M$ and $v\in T_pM$. Then  
$$\mathcal{L}_J\rho\left(p,v\right)=\Delta \left(\rho \circ u\right)
 \left(0\right),$$ 
where $u : \Delta \rightarrow \left(M,J\right)$ is any $J$-holomorphic
 disc satisfying 
$u\left(0\right)=p$ and $d_0u\left(\partial/\partial_x\right)=v$.
\end{prop}

Proposition \ref{proplevi} leads to the following proposition-definition: 
\begin{prop}\label{iproplevi2}
The two statements  are equivalent:
\begin{enumerate}
\item $\rho \circ u$ is subharmonic for any $J$-holomorphic disc $u :
 \Delta \rightarrow M$.
\item $\mathcal{L}_{J}\rho(p,v)\geq 0$ for every $p \in M$ and every $v
 \in T_pM$. 
\end{enumerate}
\end{prop}
If one of the previous statements is satisfied we say that $\rho$ is
 {\it $J$-plurisubharmonic}. We say that $\rho$ is {\it strictly $J$-plurisubharmonic} if
  $\mathcal{L}_{J}\rho(p,v)$ is positive  
for any $p \in M$ and any $v \in T_pM\setminus\{0\}$. 
Plurisubharmonic functions play a very important role in almost complex geometry:  they give attraction and 
localization properties for pseudoholomorphic discs. For this reason the 
construction of $J$-plurisubharmonic functions is crucial. 

\vspace{0,3cm}

Similarly to the integrable case, one may define the notion of pseudoconvexity in almost complex manifolds.
Let $D$ be a domain in $\left(M,J\right)$. We denote by $T^{J}\partial D:=T\partial D\cap JT\partial D$ the
 $J$-invariant subbundle of $T\partial D.$

\begin{defi}\mbox{ }

\begin{enumerate}
\item The domain $D$ is   $J$-pseudoconvex (resp. it strictly
 $J$-pseudoconvex) 
if  $\mathcal{L}_{J}\rho(p,v)\geq 0$ (resp. $>0$) for any 
$p \in \partial D$ and $v \in T^J_p\partial D$ (resp. $v \in
 T^J_p\partial D \setminus\{0\}$).
\item A $J$-pseudoconvex region is a domain $D=\{\rho<0\}$ where $\rho$
 is a $\mathcal{C}^2$ defining function,
$J$-plurisubharmonic on a neighborhood of $\overline{D}$. 
\end{enumerate}
\end{defi}
We recall that a defining function for $D$ satisfies $d\rho\neq 0$ on $\partial D$.  

\vspace{0,3cm}

We need the following lemma due to E.Chirka \cite{ch2}.

\begin{lem}\label{3lemchir}
Let $J$ be an almost complex structure of class $\mathcal{C}^1$ defined in the unit ball $\B$ of 
$\R^{2n}$ satisfying $J(0)=J_{st}$. Then there exist  positive constants $\varepsilon$ and  $A_\varepsilon=O(\varepsilon)$ 
such that the function ${\rm log} \|z\|^2+A_\varepsilon\|z\|$ is $J$-plurisubharmonic on $\B$ whenever  
$\|J-J_{st}\|_{\mathcal{C}^{1}(\overline\B)}\leq \varepsilon$. 
\end{lem}

\begin{proof}[Proof.]
This is due to the fact that for $p \in \B$ and  $\|J-J_{st} \|_{\mathcal{C}^{1}(\overline\B)}$ sufficiently small, we have: 
\begin{eqnarray*}
\mathcal{L}_JA\|z\|(p,v) & \geq & A\Big(\frac{1}{\|p\|}-\frac{2}{\|p\|}\|J(p)-J_{st}\|\\
&&\\
& & -2(1+\|J(p)-J_{st}\|) \|J-J_{st}\|_{\mathcal{C}^1(\overline\B)}\Big)\|v\|^2\\
\\
&\geq &\frac{A}{2\|p\|}\|v\|^2
\end{eqnarray*}
and
\begin{eqnarray*}
\mathcal{L}_J\ln\|z\|(p,v) &\geq &\Big(-\frac{2}{\|p\|^2}\|J(p)-J_{st}\|-
\frac{1}{\|p\|^2}\|J(p)-J_{st}\|^2-\frac{2}{\|p\|}\|J-J_{st}\|_{\mathcal{C}^1(\overline\B)} \\
&&\\
& & -\frac{2}{\|p\|}\|J(p)-J_{st}\|\|J-J_{st}\|_{\mathcal{C}^1(\overline\B)} \Big)\|v\|^2\\
&&\\
&\geq &-\frac{6}{\|p\|}\|J-J_{st}\|_{\mathcal{C}^1(\overline\B)}\|v\|^2.\\
\end{eqnarray*}
So taking $A=24\|J-J_{st}\|_{\mathcal{C}^1(\overline\B)}$ the Chirka's lemma follows.
\end{proof}

\vspace{0,7cm}

The strict $J$-pseudoconvexity of a relatively compact domain $D$ implies that there is a constant $C\geq 1$ such that:
\begin{equation}\label{3eqlev}
\frac{1}{C}\|v\|^2\leq \mathcal{L}_{J}\rho(p,v) \leq C\|v\|^2,
\end{equation}
for $p \in \partial D$ and $v \in T^J_p(\partial D)$. 

\vspace{0,7cm}

Let $\rho$ be a defining function for $D$, $J$-plurisubharmonic on a neighborhood of $\overline{D}$ and strictly $J$-plurisubharmonic
on a neighborhood of the boundary $\partial D$. Consider the one-form $d^c_J\rho$ defined by (\ref{ieqdc}) 
and let $\alpha$ be its restriction on the tangent bundle $T\partial D$. It follows that $T^{J}\partial D={\rm Ker} \alpha$. 
Due to the strict $J$-pseudoconvexity of $\rho$, the two-form $\omega:=dd^c_J\rho$ is a symplectic form (ie nondegenerate and closed) 
on a neighborhood of $\partial D$, that tames $J$. This implies that  
\begin{equation}\label{3ieqriem}
g_R:=\frac{1}{2}(\omega(.,J.)+\omega(J.,.))
\end{equation}
defines a Riemannian metric. 
We say that $T^J\partial D$ is a {\it contact structure}  and  $\alpha$ is  {\it contact form} 
for $T^J\partial D$. Consequently  vector fields in  $T^J\partial D$ span the whole tangent bundle
$T\partial D$. Indeed if $v\in T^{J}\partial D$, it follows that $\omega(v,Jv)=\alpha([v,Jv])>0$ and thus 
 $[v,Jv] \in  T\partial D \setminus T^{J}\partial D$. 
We point out that in case  $v\in T^{J}\partial D$, the vector fields  $v$ and $Jv$ are 
orthogonal with respect to the Riemannian metric $g_R$.

\subsection{The Kobayashi pseudometric}
The existence of local pseudoholomorphic discs proved by A.Nijenhuis and W.Woolf \cite{ni-wo}
allows to define the {\it Kobayashi-Royden pseudometric}, abusively  called the {\it Kobayashi pseudometric}, 
 $K_{\left(M,J\right)}$ for $p\in M$ and $v \in T_pM$:
\begin{eqnarray*}
\displaystyle K_{\left(M,J\right)}\left(p,v\right)&:=& \inf 
\Big\{\frac{1}{r}>0, u~: \Delta \rightarrow \left(M,J\right) 
\mbox{  $J$-holomorphic }, u\left(0\right)=p, d_{0}u\left(\partx\right)=rv\Big\}.\\
&=&\inf\Big\{\frac{1}{r}>0, u:\Delta_r\rightarrow (M,J), 
\mbox{  $J$-holomorphic }, u\left(0\right)=p, d_{0}u\left(\partx\right)=v\Big\}.
\end{eqnarray*}

Since the composition of pseudoholomorphic maps is still
 pseudoholomorphic, the 
 Kobayashi pseudometric satisfies the decreasing property: 
\begin{prop}\label{propdec}
Let $f~: \left(M',J'\right)\rightarrow \left(M,J\right)$ be a
 $\left(J',J\right)$-holomorphic map. Then for any
 $p \in M'$ and $v \in T_{p}M'$ we have 
$$ K_{\left(M,J\right)}\left(f\left(p\right),d_{p}f\left(v\right)\right)\leq K_{\left(M',J'\right)}\left(p,v\right).$$   
\end{prop}

Since the structures we consider are smooth enough, we may define  the
integrated pseudodistance $d_{\left(M,J\right)}$  of  $K_{\left(M,J\right)}$:
$$d_{\left(M,J\right)}\left(p,q\right): =\inf\left\{\int_0^1 
K_{\left(M,J\right)}\left(\gamma\left(t\right),\dot{\gamma}\left(t\right)\right)dt,
 \mbox{ }
\gamma~: [0,1]\rightarrow M, \mbox{ }\gamma\left(0\right)=p,
 \gamma\left(1\right)=q\right\}.$$
Similarly to the standard  integrable case, B.Kruglikov \cite{kr2} proved  that
the integrated pseudodistance of the Kobayashi pseudometric coincides with the  Kobayashi pseudodistance 
defined by chains of pseudholomorphic discs.  
%The pseudodistance $d_{\left(M,J\right)}$ is called the {\it Kobayashi pseudodistance}.
\vspace{0,3cm}

We now define the Kobayashi hyperbolicity:  

\begin{defi}\label{defin}\mbox{ }

\begin{enumerate}
\item The manifold $\left(M,J\right)$ is Kobayashi hyperbolic if the
 Kobayashi  pseudodistance $d_{\left(M,J\right)}$ is a distance.
\item The manifold $\left(M,J\right)$ is local Kobayashi hyperbolic at
 $p \in M$ if there exist a neighborhood $U$ of $p$ and a positive
 constant $C$ such that 
$$K_{\left(M,J\right)}\left(q,v\right)\geq C\|v\|$$ for every $q \in U$
 and every $v \in T_qM$.  
\item A Kobayashi hyperbolic manifold $\left(M,J\right)$ is complete
 hyperbolic if it is complete for the distance  $d_{\left(M,J\right)}$.
\end{enumerate}
\end{defi}

\section{Gromov hyperbolicity}

In this section we give some backgrounds about Gromov hyperbolic spaces. Furthermore, according to Z.M.Balogh and M.Bonk
 \cite{ba-bo}, proving that a domain $D$ with some curvature is Gromov hyperbolic reduces to 
providing sharp estimates for the Kobayashi metric $K_{(D,J)}$ near the boundary of $D$.  
%who prove that a domain $D$ with some curvature endow with a  Finsler pseudometric satisfying sharp estimates near the boundary is 
%Gromov hyperbolic we notice that . 
%proved that if the Kobayashi pseudometric (with respect to $J_{st}$) of a relatively compact 
%strictly pseudoconvex domain satisfies (\ref{3est2a}), then the Kobayashi distance is rough similar to the function $g$. 
%Their proof is purely metric and does not use complex analisys or geometry or compex analysis.

\subsection{Gromov hyperbolic spaces}

Let $(X,d)$ be a metric space. 
\begin{defi} 
The Gromov product of two points $x,y \in X$ with respect to the basepoint $\omega \in X$ is defined by 
$$(x|y)_\omega:=\frac{1}{2}(d(x,\omega)-d(y,\omega)-d(x,y)).$$ 
\end{defi}

The Gromov product measures the failure of the triangle inequality to be an equality and is always nonnegative.
%Figure 5 provides a geometric interpretation of the Gromov product of $x,y$ with respect to $\omega$ in the Euclidean plane.  
%The Gromov product of $x,y$ with respect to $\omega$ satisfies $(x|y)_\omega = \|x'-\omega\| = \|y'-\omega\|$.  

%\bigskip
%\begin{center}
%\input{gropro.pstex_t}
%\end{center}
%\begin{center}
%Figure 5. 
%\end{center}
%\bigskip

\begin{defi}\label{3defigro1} 
The metric space $X$ is Gromov hyperbolic if there is a nonnegative constant $\delta$ such that for any 
$x,y,z,\omega \in X$ one has: 
\begin{equation}\label{3eqgrhy1}
(x|y)_\omega \geq \min((x|z)_\omega,(z|y)_\omega)-\delta.
\end{equation}
\end{defi}
We point out that (\ref{3eqgrhy1}) can also be written as follows: 
\begin{equation}\label{3eqgrhy2}
d(x,y)+d(z,\omega)\leq \max(d(x,z)+d(y,\omega),d(x,\omega)+d(y,z))+2\delta,
\end{equation}
for $x,y,z,\omega \in X$.

\vspace{0,5cm}

There is a family of metric spaces for which Gromov hyperbolicity may be defined by means of geodesic triangles. 
A metric space (X,d) is said to be {\it geodesic space} if any  two points $x,y \in X$ can be joined by a {\it geodesic segment}, 
that is the image of an isometry $g~: [0,d(x,y)] \rightarrow X$ with $g(0)=x$ and $g(d(x,y))=y$. 
Such a segment is denoted by $[x,y]$. A {\it geodesic triangle} in $X$ is the subset $[x,y]\cup [y,z] \cup [z,x]$, where 
$x,y,z \in X$. For a geodesic space $(X,d)$, one may define equivalently (see \cite{gh-ha}) the Gromov hyperbolicity as follows:
\begin{defi}\label{3defigro2} 
The geodesic space $X$ is Gromov hyperbolic if there is a nonnegative constant $\delta$ such that for any 
geodesic triangle $[x,y]\cup [y,z] \cup [z,x]$  and any $\omega \in [x,y]$ one has 
$$d(\omega,[y,z] \cup [z,x])\leq \delta.$$
\end{defi}

\subsection{Gromov hyperbolicity of strictly pseudoconvex domains  in almost complex manifolds of dimension four}
Let $D=\{\rho<0\}$ be a relatively compact $J$-strictly pseudoconvex smooth domain  
in an almost complex manifolds $(M,J)$  of dimension four.
Although the boundary of a compact complex manifold with pseudoconvex boundary  is always connected, 
this is not the case in almost 
complex setting. Indeed D.McDuff  obtained in \cite{mcd} a compact almost complex manifold $(M,J)$ 
of dimension four, with a disconnected  $J$-pseudoconvex boundary. Since $D$ is globally defined by a smooth 
 function, $J$-plurisubharmonic on a neighborhood of $\overline{D}$ and strictly $J$-plurisubharmonic
on a neighborhood of the boundary $\partial D$, it follows that the boundary $\partial D$ of $D$ is connected. 
%According to Lemma \ref{} we may suppose that $D=\{\rho<0\}$, where $\rho$ is a $\mathcal{C}^{3}$ defining 
%function for $D$, $J$-plurisubharmonic in a neighborhood of the closure $\overline{D}$ of $D$. 
Moreover this also implies that there are no $J$-complex line contained in $D$ and so that $(D,d_{D,J})$ is a metric space.

\vspace{0,5cm}

A $\mathcal{C}^1$ curve $\alpha~: [0,1] \rightarrow \partial D$ is {\it horizontal} if 
$\dot{\alpha}(s) \in T_{\alpha(s)}^J\partial D$ for every $s\in [0,1]$.  
This is equivalent to $\dot{\alpha}_n\equiv 0$.  
Thus we define the {\it Levi length} of a horizontal curve by 
\begin{equation*}
\displaystyle \mathcal L_J\rho-{\rm length}(\alpha):= \int_0^1 \mathcal L_J\rho(\alpha(s),\dot{\alpha}(s))^{\frac{1}{2}}ds.
\end{equation*}
We point out that, due to (\ref{3ieqriem}), 
\begin{equation*}
\displaystyle \mathcal L_J\rho-{\rm length}(\alpha)= \int_0^1 g_R(\alpha(s),\dot{\alpha}(s))^{\frac{1}{2}}ds.
\end{equation*}
Since  $T^J\partial D$ is a {\it contact structure}, a theorem due to Chow \cite{ch} states that any two points in $\partial D$
may be connected by a  $\mathcal{C}^1$ horizontal curve. This allows to define 
the {\it Carnot-Carath\'eodory metric } as follows:
\begin{equation*}
\displaystyle d_{H}(p,q):= \left\{ \mathcal L_J\rho-{\rm length}(\alpha), 
\alpha~: [0,1] \rightarrow \partial D \mbox{ }\mbox{ horizontal }\mbox{ }, \alpha(0)=p, \alpha(1)=q \right\}.
\end{equation*} 
%According to (\ref{3eqlev}) the Carnot-Carath\'eodory metric is equivalent up to a constant to the following metric 
%\begin{equation}
%\displaystyle d_{h}(p,q):= \left\{  \int_0^1 \|\dot{\alpha}_t(s)\|ds, 
%\alpha~: [0,1] \rightarrow \partial D \mbox{ }\mbox{ horizontal }\mbox{ }, \alpha(0)=p, \alpha(1)=q \right\}.
%\end{equation}     

\vspace{0,5cm}

Equivalently, we may define locally the {\it Carnot-Carath\'eodory metric } by means of vector fields  as follows. 
Consider two $g_R$-orthogonal vector fields 
$v,Jv \in  T^J\partial D$ and the {\it sub-Riemannian metric} associated to $v,Jv$:
$$g_{SR}(p,w):=\inf \left\{a_1^2+a_2^2, \mbox{ } a_1v(p)+a_2(Jv)(p)=w\right\}.$$   
For a horizontal curve $\alpha$, we set 
\begin{equation*}
\displaystyle g_{SR}-{\rm length}(\alpha):= \int_0^1 g_{SR}(\alpha(s),\dot{\alpha}(s))^{\frac{1}{2}}ds.
\end{equation*}
Thus we define: 
\begin{equation*}
\displaystyle d_{H}(p,q):= \left\{g_{SR}-{\rm length}(\alpha), 
\alpha~: [0,1] \rightarrow \partial D \mbox{ }\mbox{ horizontal }\mbox{ }, \alpha(0)=p, \alpha(1)=q \right\}.
\end{equation*}
We point out that for a  small horizontal curve $\alpha$, we have 
$$\dot{\alpha}(s)=a_1(s)v(\alpha(s))+a_2(s)J(\alpha(s))v(\alpha(s)).$$
Consequently 
$$g_{R}(\alpha(s),\dot{\alpha}(s))=\left[a_1^2(s)+a_2^2(s)\right]g_R(\alpha(s),v(\alpha(s))).$$ 
Although the role of the bundle $T^{J}\partial D$ is crucial, it is not essential to define the Carnot-Carath\'eodory 
metric with $g_{SR}$ instead of $g_R$. Actually, two 
Carnot-Carath\'eodory metrics defined with different Riemannian metrics are bi-Lipschitz equivalent (see \cite{gr3}).

\vspace{0,5cm}

According to A.Bellaiche \cite{bel} and M.Gromov \cite{gr3} and since $T\partial D$ is spanned by  
vector fields of $T^J\partial D$ and  Lie Brackets of vector fields of $T^J\partial D$, balls with respect to the 
Carnot-Carath\'eodory metric may be anisotropically approximated. More precisely
\begin{prop}
There exists a positive constant $C$ such that for $\varepsilon$  small enough and $p \in \partial D$:
\begin{equation}\label{3eqboxball}
{\rm Box}\left(p,\frac{\varepsilon}{C}\right) \subseteq \B_H(p,\varepsilon) \subseteq {\rm Box}(p,C\varepsilon),
\end{equation}
where $\B_H(p,\varepsilon):=\{q \in \partial D, d_H(p,q)<\varepsilon \}$ and 
$ {\rm Box}(p,\varepsilon):=\{p+v \in \partial D, |v_t|<\varepsilon, |v_n|<\varepsilon^2\}$.
\end{prop}
The splitting $v=v_t+v_n$ is taken at $p$. We point out that choosing local coordinates such that 
$p=0$, $J(0)=J_{st}$ and $T^J_0\partial D=\{z_1=0\}$, then 
 $ {\rm Box}(p,\varepsilon)=\partial D \cap Q(0,\epsilon)$, where $Q(0,\epsilon)$ is the classical polydisc
$ Q(0,\epsilon):=\{z \in \C^2, |z_1|<\varepsilon^2, |z_2|<\varepsilon \}$.

As proved by Z.M.Balogh and M.Bonk \cite{ba-bo}, (\ref{3eqboxball}) allows to approximate the Carnot-Carath\'eodory metric 
by a Riemannian anisotropic metric:
\begin{lem}
There exists a positive constant $C$ such that for any positive $\kappa$ 
$$\frac{1}{C}d_\kappa(p,q) \leq d_H(p,q)\leq Cd_\kappa(p,q),$$
whenever  $d_H(p,q)\geq 1/\kappa$ for $p,q \in \partial D$. Here, the distance $d_\kappa(p,q)$ is taken with respect to 
the Riemannian  metric $g_\kappa$ defined by:
$$g_\kappa(p,v):= \mathcal{L}_J\rho(p,v_h)+\kappa^2|v_n|^2,$$
for $p\in \partial D$  and $v=v_t+v_n \in T_p\partial D$.
\end{lem}

\vspace{0,5cm}

The crucial idea of Z.M.Balogh and M.Bonk \cite{ba-bo} to prove the Gromov hyperbolicity of $D$ is to 
introduce a function on $D\times D$, using the Carnot-Carath\'eodory metric, which satisfies $(\ref{3eqgrhy1})$ and which is
roughly similar to the Kobayashi distance. 

For  $p \in D$ we define  a boundary projection map $\pi~: D \rightarrow \partial D$ by 
$$\delta(p)=\|p-\pi(p)\|={\rm dist}(p,\partial D).$$ 
We notice that $\pi(p)$ is  uniquely determined only if $p \in D$ is sufficiently close to the boundary. 
%In that case $\pi(p)=p^*$.
We set 
$$h(p):=\delta(p)^{\frac{1}{2}}.$$
Then we define a map $g~:D\times D \rightarrow [0,+\infty)$ by:
\begin{equation*} 
g(p,q):=2\log \left(\frac{d_H(\pi(p),\pi(q))+\max\{h(p),h(q)\}}{\sqrt{h(p)h(q)}}\right),
\end{equation*}
for $p,q \in D$.
The map $\pi$ is uniquely determined only near the boundary. But an other choice of $\pi$ 
 gives a function $g$ that coincides up to a bounded  additive constant that will not disturb our results. 
The motivation of introducing the map $g$ is related with the Gromov hyperbolic space
${\rm Con}(Z)$ defined by M.Bonk and O.Schramm 
in \cite{bo-sc} (see also \cite{gr2}) as follows. 
Let $(Z,d)$ be a bounded metric space which does not consist of a single point and set 
$${\rm Con}(Z):=Z\times (0,diam(Z)].$$ 
Let us define a map $\widetilde{g}: {\rm Con}(Z)\times {\rm Con}(Z) \rightarrow [0,+\infty)$ by 
$$\widetilde{g}\left((z,h),(z',h')\right):=2\log \left(\frac{d(z,z')+\max\{h,h'\}}{\sqrt{hh'}}\right).$$
M.Bonk and O.Schramm  in \cite{bo-sc} proved that $({\rm Con}(Z),\widetilde{g})$ is a Gromov hyperbolic (metric) space. 

In our case the map $g$ is not a metric on $D$ since two different points $p\neq q \in D$ may have the same projection; nevertheless
\begin{lem} 
The function $g$ satisfies  (\ref{3eqgrhy2}) (or equivalently (\ref{3eqgrhy1})) on $D$.
\end{lem}
\begin{proof}[Proof.]
Let $r_{ij}$ be real nonnegative numbers such that 
$$r_{ij}=r_{ji} \mbox{ } \mbox{ and } \mbox{ } r_{ij}\leq r_{ik}+r_{kj},$$
for $i,j,k=1,\cdots,4$. Then 
\begin{equation}\label{3eqgro}
r_{12}r_{34}\leq 4 \max (r_{13}r_{24}, r_{14}r_{23}).
\end{equation}

Consider now four points $p_i \in D$, $i=1,\cdots,4$. We set $h_i=\delta(p_i)^{\frac{1}{2}}$ and 
$d_{i,j}=d_{(H,J)}(\pi(p_i),\pi(p_j))$. Then applying  (\ref{3eqgro}) to $r_{ij}=d_{i,j}+\min (h_i,h_j)$, we obtain: 

\vspace{0.3cm}

$ (d_{1,2}+\min (h_1,h_2)) (d_{3,4}+\max (h_3,h_4))$

\vspace{0.2cm}

\hspace{1cm}$\leq 4\max((d_{1,3}+\max (h_1,h_3))(d_{2,4}+\min (h_2,h_4),
(d_{1,4}+\min (h_1,h_4))(d_{2,3}+\max (h_2,h_3)).$

\vspace{0.3cm}
Then: 
$$g(p_1,p_2)+g(p_3,p_4)\leq \max(g(p_1,p_3)+g(p_2,p_4),g(p_1,p_4)+g(p_2,p_3))+2\log 4,$$
which proves the desired statement.
\end{proof}

As a direct corollary, if a metric $d$ on $D$ is roughly similar 
to $g$, then the metric space $(D,d)$ is Gromov hyperbolic:
\begin{cor}\label{3corgrom}
Let $d$ be a metric on $D$ verifying 
\begin{equation}\label{3eqmet}
-C+g(p,q) \leq d(p,q) \leq g(p,q)+C
\end{equation}
for some  positive constant $C$, and every $p,q \in D$. Then  $d$ satisfies (\ref{3eqgrhy2}) and so 
the metric space $(D,d)$ is Gromov hyperbolic.   
\end{cor} 

\vspace{0.5cm}

Z.M.Balogh and M.Bonk \cite{ba-bo} proved that if the Kobayashi metric (with respect to $J_{st}$) of a bounded 
strictly pseudoconvex domain satisfies (\ref{3est2a}), then the Kobayashi distance is rough similar to the function $g$. 
Their proof is purely metric and does not use complex geometry or complex analysis. We point out that 
the strict pseudoconvexity is only needed to obtain (\ref{3eqlev}) or the fact that $T\partial D$ is spanned by 
vector fields of $T^{J_{st}}\partial D$ and Lie Brackets of vector fields of $T^{J_{st}}\partial D$.   
In particular their proof remains valid in the almost complex setting and, consequently, Theorem A implies:
\begin{theo}\label{3corgro}
Let $D$ be a relatively compact strictly $J$-pseudoconvex smooth domain 
in an almost complex manifold $(M,J)$  of dimension four. There is a nonnegative constant $C$ 
such that for any $p,q \in D$ 
$$g(p,q)-C \leq d_{(D,J)}(p,q)\leq g(p,q)+C.$$
\end{theo}
According to Corollary \ref{3corgrom} we finally obtain the following theorem (see also (1) of Theorem B):
\begin{theo}\label{3theogro}
Let $D$ be a relatively compact strictly $J$-pseudoconvex smooth domain 
in an almost complex manifolds $(M,J)$  of dimension four. Then the metric space $(D,d_{(D,J)})$ is Gromov 
hyperbolic.
\end{theo}

\begin{ex}\label{3ex1}
There exist a neighborhood $U$ of $p$ and a diffeomorphism $z:U \rightarrow \B \subseteq \R^4$,
centered at $p$, such that the function $\|z\|^2$ is strictly
$J$-plurisubharmonic on $U$ and 
$\|z_*(J)-J_{st}\|_{\mathcal C^2(U)} \leq \lambda_0$. 
Hence the unit ball $\B$ equipped with the  metric $\displaystyle d_{(\B(0,1),z_*J)}$ is Gromov hyperbolic. 
\end{ex}
As a direct corollary of Example \ref{3ex1}  we have (see also (2) of Theorem B):
\begin{cor}\label{3cor3}
Let $(M,J)$ be a four dimensional almost complex manifold. Then every point $p \in M$ has a 
basis of Gromov hyperbolic neighborhoods.
\end{cor}

%\subsection{Fefferman's mapping theorem on four dimensional almost complex manifolds}

%\begin{cor}
%%Let $D$ and $D'$ be two relatively compact strictly  pseudoconvex domains with $\mathcal{C}^3$ boundary 
%in four dimensional almost complex manifolds $(M,J)$ and   $(M',J')$. Then a smooth
%diffeomorphism $f:  D  \rightarrow  D'$ extends continuously to a 
%diffeomorphism between $\bar D$ and $\bar D'$.
%\end{cor}

\section{Sharp estimates of the Kobayashi metric}
In this section we give a precise localization principle for the Kobayashi metric and we prove Theorem A. 

Let $D=\{\rho<0\}$ be a  domain in an almost complex manifold $(M,J)$, where $\rho$ is a smooth  defining strictly
$J$-plurisubharmonic function. 
For a point $p \in D$ we define 
\begin{equation}
\delta(p):={\rm dist}(p,\partial D),
\end{equation}
and for $p$ sufficiently close to $\partial D$, we define $\pi(p)\in  \partial D$ as the unique boundary 
point such that:
\begin{equation}
\delta(p)=\|p-\pi(p)\|.
\end{equation}
For $\varepsilon>0$, we  introduce
\begin{equation}
N_\varepsilon:=\{p \in D, \delta(p)<\varepsilon \}.
\end{equation}

\subsection{Sharp localization principle}
F.Forstneric and  J.-P.Rosay \cite{fo-ro} obtained a sharp localization principle of the Koba\-yashi metric near 
a strictly $J_{st}$-pseudoconvex boundary point of a domain $D\subset \C^n$. However their approach is based on the 
existence of  some holomorphic peak function at such a  point; this is purely complex and cannot be generalized in the 
nonintegrable case. The sharp localization principle we give is based on some estimates of the Kobayashi length of 
a path near the boundary. 
\begin{prop}\label{3thm}
There exists  a positive constant $r$ such that for every $p \in D$ sufficiently close to the boundary and for every sufficiently small 
neighborhood $U$ of $\pi(p)$ there is a positive constant $c$ such that for every 
  $v\in T_pM$: 
\begin{equation}\label{3est1}
 K_{(D\cap U,J)}(p,v)\geq  (1-c \delta(p)^r) K_{(D \cap U,J)}(p,v).
\end{equation}
\end{prop}
We will give later a more precise version of Proposition \ref{3thm}, where the constants $c$ and $r$ are given 
explicitly (see Lemma \ref{3lemlocquant}).
\begin{proof}
We consider a local diffeomorphism $z$ centered at $\pi(p)$ from a sufficiently small neighborhood $U$ of $\pi(p)$ to $z(U)$ such that 
\begin{enumerate}
\item $z(p)=(\delta(p),0)$, 
\item the structure $z_*J$ satisfies $z_*J(0)=J_{st}$ and is diagonal, 
%\begin{equation}
%(z_*J)_\C=\left(\begin{array}{cccc} 
%a_{1} & \overline{b_{1}} & 0 & 0 \\
%b_{1} & \overline{a_{1}} & 0 & 0 \\
%0 & 0 & a_{2} & \overline{b_{2}} \\
%0 & 0 &  a_{2} & \overline{a_{2}}\\ 
%\end{array}\right), 
%\end{equation}
\item the defining function $\rho\circ z^{-1}$ is locally expressed by:  
$$\rho\circ z^{-1} \left(z\right)=
-2\Re e z_1+2\Re e \sum \rho_{j,k} z_jz_k+ \sum \rho_{j,\overline{k}} z_j\overline{z_k}+O(\|z\|^3),$$  
where $\rho_{j,k}$ and $\rho_{j,\overline{k}}$ are constants satisfying $\rho_{j,k}=\rho_{k,j}$ and  
$\rho_{j,\overline{k}}=\overline{\rho}_{k,\overline{j}}$.
\end{enumerate}
According to Lemma  $4.8$ in \cite{le}, there exists a positive constant 
$c_1$ ($C_{1/4}$ in the notations of \cite{le}), independent of $p$,  such that, shrinking $U$ if necessary, for any $q \in D\cap U$ and any 
$v \in T_q \R^4$: 
$$K_{(D,J)}(q,v)\geq c_1 \frac{\|d_q\chi(v)\|}{\chi(q)},$$
where  $\chi(q):=|z_1(q)|^2+|z_2(q)|^4$.

%Let us prove that there exists a positive constant $c$ and a neighborhood  $V\subset U$ of $p$, 
%such that  for any $J$-holomorphic disc 
%$u~: \Delta \rightarrow D$ satisfying $u(0)=p \in D \cap V$ we have 
%$$u(\Delta_s)\subset D\cap U$$
%with 
%$$s:= 1-c\delta(p)^{c_1}.$$
Let $u~: \Delta \rightarrow D$ be a $J$-holomorphic discs satisfying $u(0)=p \in D$. 
Assume that $u(\Delta)\not\subset D\cap U$ and let $\zeta\in \Delta$ such that $u(\zeta) \in D\cap \partial U$.  We consider a $\mathcal{C^1}$ path 
$\gamma:[0;1]\to D$ from $u(\zeta)$ to the point $p$; so $\gamma (0)=u(\zeta)$ and $\gamma (1)=p$. Without loss of generality we may  suppose 
that $\gamma([0,1[) \subseteq D\cap U$. 
From this we get that the Kobayashi length of $\gamma$ satisfies:
\begin{eqnarray*}
L_{(D,J)}(\gamma) &:= &\int_0^1K_{(D,J)}(\gamma(t),\dot\gamma(t))\mathrm{d}t\\
&&\\
& \geq & c_1 \int_0^1  \frac{\|d_{\gamma(t)}\chi(\dot{\gamma}(t))\|}{\chi(\gamma(t))}dt.
\end{eqnarray*}
This leads to: 
\begin{eqnarray*}
L_{(D,J)}(\gamma) & \geq &    c_1 \int_{\chi(p)}^{\chi(u(s\zeta))} \frac{dt}{t}=c_1 \left|\log\frac{\chi(u(s\zeta))}{\chi(p)}\right|
=c_1 \log\frac{\chi(u(s\zeta))}{\chi(p)},
\end{eqnarray*}
for $p$ sufficiently small.
Since there exists a positive constant $c_2(U)$ such that for all $z \in D\cap \partial U$:
$$\chi(z)\geq c_2(U),$$
and since $\chi(p)=\delta(p)^2$
it follows that
\begin{equation}\label{3eqchemin}
L_{(D,J)}(\gamma)  \geq  c_1\log \frac{c_2(U)}{\delta(p)^2},
\end{equation} 
We set $c_3(U)=c_1\log(c_2(U))$. 

According to the decreasing property of the Kobayashi distance, we have: 
\begin{equation}\label{3eqkob2110}
d_{(D,J)}(p,u(\zeta))\leq d_{(\Delta,J_{st})}(0,\zeta)= \log\frac{1+|\zeta|}{1-|\zeta|}.
\end{equation}
Due to (\ref{3eqchemin}) and (\ref{3eqkob2110}) we have: 
\begin{equation*}
\frac{e^{c_3(U)}- \delta(p)^{2c_1}}{e^{c_3(U)}+\delta(p)^{2c_1}}\leq |\zeta|,
\end{equation*}
and so for $p$ sufficiently close to its projection point $\pi(p)$:  
\begin{equation*}
1-2e^{-c_3(U)}\delta(p)^{2c_1} \leq |\zeta|,
\end{equation*}
This finally proves that  
$$u(\Delta_s)\subset D\cap U$$
with $s:= 1-2e^{-c_3(U)}\delta(p)^{2c_1}.$
\end{proof}

%\begin{rem}
%Notice that the positive constant $c_3$ in the previous proof describe the behaviour of the neighborhood $U$.
%This will be useful in what follows in order to obtain a quantitative 
%localization principle with constants independent of the points  $p$ and $\pi(p)$.%
%\end{rem}

\subsection{Sharp estimates of the Kobayashi metric}

In this subsection we give the proof of Theorem A.  
\begin{proof}
Let $p \in D\cap N_{\varepsilon_0}$ with $\varepsilon_0$ small enough and set $\delta:=\delta(p)$.  Considering  a local diffeomorphism $z:U\rightarrow z(U)\subset \R^4$ such that Proposition 
\ref{3thm} holds, me may  assume that:   
\begin{enumerate}
\item $\pi(p)=0$ and  $p=(\delta,0)$. 
\item $D\cap U \subset \R^4$,
\item The structure $J$ is diagonal and coincides with $J_{st}$ on the complex tangent space $\{z_1=0\}$:
\begin{equation}
J_\C=\left(\begin{array}{cccc} 
a_{1} & \overline{b_{1}} & 0 & 0 \\
b_{1} & \overline{a_{1}} & 0 & 0 \\
0 & 0 & a_{2} & \overline{b_{2}} \\
0 & 0 &  a_{2} & \overline{a_{2}}\\ 
\end{array}\right), 
\end{equation}
with \begin{equation*}
\left\{
\begin{array}{lll}  
a_{l}&=&i+O(\|z_1\|^2),\\
\\
b_{l}&=&O(\|z_1\|),
\end{array}
\right.
\end{equation*}
for $l=1,2$.
\item The defining function $\rho$ is expressed by:  
$$\rho \left(z\right)=
-2\Re e z_1+2\Re e \sum \rho_{j,k} z_jz_k+ \sum \rho_{j,\overline{k}} z_j\overline{z_k}+O(\|z\|^3),$$  
where $\rho_{j,k}$ and $\rho_{j,\overline{k}}$ are constants satisfying $\rho_{j,k}=\rho_{k,j}$ and  
$\rho_{j,\overline{k}}=\overline{\rho}_{k,\overline{j}}$.
\end{enumerate}

\vspace{0.5cm}
Since the structure $J$ is diagonal, the Levi form of $\rho$ at the origin with
respect to the structure $J$ coincides with the Levi form of $\rho$ at the origin with respect to the
structure $J_{st}$ on the complex tangent space.  It follows essentially from \cite{ga-su}. 
\begin{lem}\label{3lemlev0}
Let $v_2=(0,v_2) \in \R^4$ be a tangent vector to $\partial D$ at the origin. We have:
\begin{equation}\label{3eqeas1}
\rho_{2,\overline{2}}|v_2|^2=
\mathcal{L}_{J_{st}}\rho(0,v_2)=\mathcal{L}_{J}\rho(0,v_2).
\end{equation}
\end{lem}
\begin{proof}[Proof of Lemma \ref{3lemlev0}.]
Let $u:\Delta \rightarrow \C^2$ be a $J$-holomorphic disc such that $u(0)=0$ 
and tangent to $v_2$,  
$$u(\zeta) = \zeta v_2 + \mathcal O(|\zeta|^2).$$  
Since $J$ is a diagonal structure, the $J$-holomorphy equation leads to:
\begin{equation}\label{3eqholo}
\frac{\partial u_1}{\partial \overline{\zeta}}= q_1(u)\overline{\frac{\partial u_1}{\partial \zeta}},
\end{equation}
where $q_1(z)=O(\|z\|)$. Moreover, since $\displaystyle d_0 u_1=0$, (\ref{3eqholo}) gives:
$$\frac{\partial^2 u_1}{\partial \zeta \partial \overline{\zeta}}(0) = 0.$$ 
This implies that 
$$\frac{\partial^2 \rho \circ u}{\partial \zeta \partial \overline{\zeta}}(0)= \rho_{2,\overline{2}}|v_t|^2.$$ 
Thus, the Levi form with respect to $J$ coincides with the Levi form with respect to $J_{st}$ on the
 complex tangent space of $\partial D^\delta$ at the origin.
\end{proof}
\begin{rem}
More generally, even if $J(0) = J_{st}$, the Levi form of a function $\rho$ with respect to $J$ at the origin
does not coincide with the Levi form of $\rho$ with respect to $J_{st}$. According to Lemma \ref{3lemlev0} if the structure is diagonal then 
they are equal at the origin on the complex tangent space; but in real dimension greater than four, the structure can not be (genericaly) 
diagonal. K.Diederich and A.Sukhov \cite{di-su} proved that if the structure $J$  satisfies $J(0)=J_{st}$ and $d_zJ=0$ 
(which is always possible by a local diffeomorphism in arbitrary dimensions), 
then the Levi forms coincide at the origin (for all the directions).
\end{rem}
Lemma  \ref{3lemlev0} implies that since the domain $D$ is strictly $J$ pseudoconvex at $\pi(p)=0$, 
we may assume that $\rho_{2,\overline{2}}=1$.  

\vspace{0.7cm}

%\subsection*{Lower estimate}

Consider the following biholomorphism $\Phi$ (for the standard structure $J_{st}$) that removes the harmonic term 
$2\Re e (\rho_{2,2}z_2^2)$: 
\begin{equation}\label{3eqharm}
\Phi(z_1,z_2):=(z_1-\rho_{2,2}z_2^2,z_2).
\end{equation}
The complexification of the structure $\Phi_*J$ admits the following matricial 
representation: 
\begin{equation}\label{3eqcomp}
(\Phi_*J)_\C=\left(\begin{array}{cccc} 
a_{1}(\Phi^{-1}(z)) & \overline{b_{1}(\Phi^{-1}(z))} & c_1(z) & \overline{c_2(z)} \\
b_{1}(\Phi^{-1}(z)) & \overline{a_{1}(\Phi^{-1}(z))} & c_2(z)  & \overline{c_1(z)} \\
0 & 0 & a_{2}(\Phi^{-1}(z)) & \overline{b_{2}(\Phi^{-1}(z))} \\
0 &   0 &  b_{2}(\Phi^{-1}(z)) & \overline{a_{2}(\Phi^{-1}(z))} \\ 
\end{array}\right), 
\end{equation}
where 
\begin{equation*}
\left\{
\begin{array}{lll}  
c_1(z)&:=& 2\rho_{2,2} z_2\left(a_{1}(\Phi^{-1}(z)) -a_{2}(\Phi^{-1}(z)\right) \\
\\
c_2(z)&:=& 2\rho_{2,2} z_2b_{1}(\Phi^{-1}(z))-\overline{\rho_{2,2} z_2}b_{2}(\Phi^{-1}(z)).
\end{array}
\right.
\end{equation*}

\vspace{0.5cm}

In what follows, we need a quantitative version of Proposition \ref{3thm}. 
So we consider the following polydisc  
$Q_{(\delta,\alpha)}:=\{z \in \C^2, |z_1|<\delta^{1-\alpha}, |z_2|<c\delta^{\frac{1-\alpha}{2}}\}$ centered at the origin, 
where $c$ is chosen such that 
\begin{equation}\label{3eqbord}
\Phi(D\cap U)\cap \partial Q_{(\delta,\alpha)} \subset \{z\in \C^{2}, |z_1|=\delta^{1-\alpha}\}. 
\end{equation}
\begin{lem}\label{3lemlocquant}
Let $0<\alpha<1$ be a positive number. There is a  positive constant $\beta$ such that for every sufficiently small $\delta$ we have: 
\begin{equation}\label{3eqloc}
K_{(D\cap U,J)}(p,v)=K_{(\Phi(D\cap U),\Phi_*J)}(p,v)\geq 
\left(1-2\delta^\beta\right)K_{(\phi(D\cap U)\cap Q_{(\delta,\alpha)},\Phi_*J)}(p,v),
\end{equation}
for  $p=(\delta,0)$  and  every $v\in T_p \R^4$.
\end{lem}
\begin{proof}
The proof is a quantitative repetition of the proof of  Proposition \ref{3thm}; 
we only notice that according to (\ref{3eqbord}) we have 
$c_2=\delta^{1-\alpha}$, implying  $\beta=2\alpha c_1$.
\end{proof}

\vspace{0.5cm}

Let $0<\alpha<\alpha'<1$ to be fixed later, independently of $\delta$.
For every sufficiently small $\delta$, we consider a smooth cut off function $\chi~: \R^4 \rightarrow \R$: 
$$\left\{
\begin{array}{lll}
\chi&\equiv& 1 \mbox{ } \mbox{ on } \mbox{ } Q_{(\delta,\alpha)},\\
\\
\chi&\equiv& 0 \mbox{ } \mbox{ on } \mbox{ } \R^4\setminus Q_{(\delta,\alpha')},\\
\end{array}
\right.$$
with $\alpha'<\alpha$. We point out that $\chi$ may be chosen such that 
\begin{equation}\label{3eqderiv}
\|d_z\chi\|\leq \frac{c}{\delta^{1-\alpha'}},
\end{equation}
for some positive constant $c$ independent of $\delta$.  
We consider now the following endomorphism of $\R^4$: 
$$q'(z):=\chi(z)q(z),$$
for $z \in Q_{(\delta,\alpha')}$, where
$$q(z):=(\Phi_*J(z)+J_{st})^{-1}(\Phi_*J(z)-J_{st}).$$
According to the fact that  $q(z)=O(|z_1+\rho_{2,2}z_2^2|)$ (see (\ref{3eqcomp})) 
and according to (\ref{3eqderiv}), the differential of  $q'$ is upper bounded on  
$Q_{(\delta,\alpha')}$, independently of $\delta$.   Moreover the $\displaystyle  dz_2\otimes \frac{\partial}{\partial z_1}$ and 
the $\displaystyle dz_2\otimes \frac{\partial}{\partial \overline{z_1}}$ components of the structure $\Phi_*J$ are $O(|z_1+\rho_{2,2}z_2^2||z_2|)$ 
by (\ref{3eqcomp}); this is also the case for the endomorphism $q'$. 
We define an almost complex structure on the whole space $\R^4$ by:
$$J'(z)=J_{st}(Id+q'(z))(Id-q'(z))^{-1},$$
which is well defined since $\|q'(z)\|<1$. 
It follows that the structure $J'$ is identically equal to $\Phi_*J$ in $Q_{(\delta,\alpha)}$ 
and  coincides with $J_{st}$ on  $\R^4\setminus Q_{(\delta,\alpha')}$ (see Figure 1). 
Notice also that since  $\chi\equiv d\chi\equiv0$ on  $\partial Q_{(\delta,\alpha')}$,  
$J'$ coincides with $J_{st}$ at first  order  on $\partial Q_{(\delta,\alpha')}$. 
Finally the  structure $J'$ satisfies:
$$J'=J_{st}+O(|z_1+\rho_{2,2}z_2^2|)$$
on $Q_{(\delta,\alpha')}.$ 
To fix the notations, the almost complex structure $J'$ admits the following matricial interpretation:
\begin{equation}
 J'_\C=\left(\begin{array}{cccc} 
a'_{1} & \overline{b'_{1}} & c'_1 & \overline{c'_2} \\
b'_{1} & \overline{a'_{1}} & c'_2 & \overline{c'_1} \\
0 & 0 & a'_{2} & \overline{b'_{2}} \\
0 & 0 &  b'_{2} & \overline{a'_{2}} \\ 
\end{array}\right). 
\end{equation}
with \begin{equation*}
\left\{
\begin{array}{lll}  
a'_{l}&=&i+O(\|z\|^2),\\
\\
b'_{l}&=&O(\|z\|),\\
\\
c'_l& =& O(|z_2|\|z\|), 
\end{array}
\right.
\end{equation*}
for $l=1,2$.

\vspace{0.5cm}

\bigskip
\begin{center}
\begin{picture}(0,0)%
\includegraphics{cutt1.pstex}%
\end{picture}%
\setlength{\unitlength}{1934sp}%
\begingroup\makeatletter\ifx\SetFigFont\undefined%
\gdef\SetFigFont#1#2#3#4#5{%
  \reset@font\fontsize{#1}{#2pt}%
  \fontfamily{#3}\fontseries{#4}\fontshape{#5}%
  \selectfont}%
\fi\endgroup%
\begin{picture}(10923,7426)(661,-6920)
\put(7876,-4411){\makebox(0,0)[lb]{\smash{{\SetFigFont{10}{12.0}{\familydefault}{\mddefault}{\updefault}{\color[rgb]{0,0,0}$p=(\delta,0)$}%
}}}}
\put(6826,-4036){\makebox(0,0)[lb]{\smash{{\SetFigFont{10}{12.0}{\familydefault}{\mddefault}{\updefault}{\color[rgb]{0,0,0}$0$}%
}}}}
\put(6151,-6061){\makebox(0,0)[lb]{\smash{{\SetFigFont{10}{12.0}{\familydefault}{\mddefault}{\updefault}{\color[rgb]{0,0,0}$Q_{(\delta,\alpha)}$}%
}}}}
\put(5551,-6811){\makebox(0,0)[lb]{\smash{{\SetFigFont{10}{12.0}{\familydefault}{\mddefault}{\updefault}{\color[rgb]{0,0,0}$Q_{(\delta,\alpha')}$}%
}}}}
\put(8626,-586){\makebox(0,0)[lb]{\smash{{\SetFigFont{10}{12.0}{\familydefault}{\mddefault}{\updefault}{\color[rgb]{0,0,0}$J$}%
}}}}
\put(3751,-1561){\makebox(0,0)[lb]{\smash{{\SetFigFont{10}{12.0}{\familydefault}{\mddefault}{\updefault}{\color[rgb]{0,0,0}$J'$}%
}}}}
\put(1876,-2911){\makebox(0,0)[lb]{\smash{{\SetFigFont{10}{12.0}{\familydefault}{\mddefault}{\updefault}{\color[rgb]{0,0,0}$J_{st}$}%
}}}}
\put(676,-5686){\makebox(0,0)[lb]{\smash{{\SetFigFont{10}{12.0}{\familydefault}{\mddefault}{\updefault}{\color[rgb]{0,0,0}$J'_{|\partial Q_{(\delta,\alpha')}}=J_{st}$ at order 1}%
}}}}
\put(3526,239){\makebox(0,0)[lb]{\smash{{\SetFigFont{10}{12.0}{\familydefault}{\mddefault}{\updefault}{\color[rgb]{0,0,0}$J_{|\partial Q_{(\delta,\alpha)})}=J'_{|\partial Q_{(\delta,\alpha')}}$ at order 1}%
}}}}
\put(10051,-5986){\makebox(0,0)[lb]{\smash{{\SetFigFont{10}{12.0}{\familydefault}{\mddefault}{\updefault}{\color[rgb]{0,0,0}$\Phi(D\cap U)$}%
}}}}
\end{picture}%

\end{center}
\begin{center}
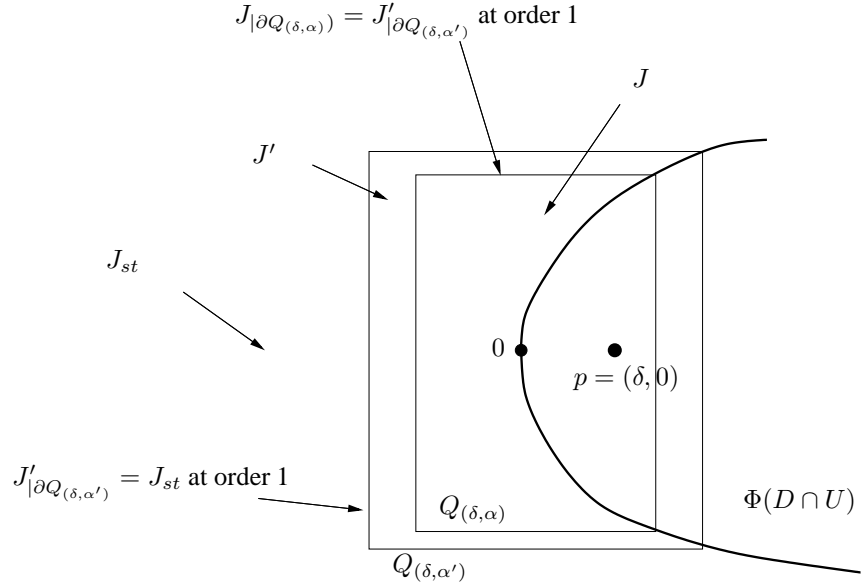
Figure 1. Extension of the almost complex structure $J$.
\end{center}

\bigskip

Furthermore, according to the decreasing property of the Kobayashi metric we have for $p=(\delta,0)$:
\begin{equation}\label{3eqloc2}
 K_{(\Phi(D\cap U) \cap Q_{(\delta,\alpha)},\Phi_*J)}(p,v) = K_{(\Phi(D\cap U) \cap Q_{(\delta,\alpha)},J')}(p,v) \geq  K_{(\Phi(D\cap U) \cap Q_{(\delta,\alpha')},J')}(p,v).
\end{equation}
Finally, (\ref{3eqloc}) and (\ref{3eqloc2}) lead to: 
\begin{equation}\label{3eqloc3}
  K_{(D\cap U,J)}(p,v) \geq (1-2\delta^\beta) K_{(\Phi(D\cap U) \cap Q_{(\delta,\alpha')},J')}(p,v).
\end{equation}
This implies that in order to obtain the lower estimate of Theorem A it is sufficient to prove lower estimates for  
$K_{(\Phi(D\cap U) \cap Q_{(\delta,\alpha')},J')}(p,v)$.

%According to the fact that  if $z \in  \{\rho \circ \Phi^{-1}(z)\leq 0\}$ then
%$$|z_2|^2+O(|z_1|\|z\|)+O(|z_2|^3)\leq 2 \Re e z_1,$$
%there is a positive constant $C''$ such that 
%$$ |z_2|^2\leq C''|z_1|,$$
%if $z$  is sufficiently small. 
%Hence we have  for a positive constant $C_1$
%\begin{equation}\label{3eqpara}
%\|z\|^2\leq C_1 |z_1|.
%\end{equation} 
%\vspace{0.5cm}

%Since $\Phi$ is a small perturbation of the identity in a neighborhood 
%of the origin, it follows  that $\Phi(Q_{(\delta,\alpha')})$ is close to  $Q_{(\delta,\alpha')}$. More precisely, 
%there is a positive constant $c$ such that 
%\begin{equation}\label{3eqinclus}
%\{z \in \C^2, |z_1|<c'\delta^{1-\alpha'}, |z_2|<c\delta^{\frac{1-\alpha'}{2}}\}\subseteq \Phi(Q_{(\delta,\alpha')}).
%\end{equation}
%This will be useful for next Lemma \ref{3lemma}.

\vspace{0,5cm}
We set $\Omega:=\Phi(D\cap U) \cap Q_{(\delta,\alpha')}$. 
Let $T_\delta$ be the translation of $\C^2$ defined by 
$$T_\delta(z_1,z_2):=(z_1-\delta,z_2),$$
and let $\varphi_{\delta}$ be a linear diffeomorphism  of $\R^4$ such that 
the direct image of $J'$ by $\varphi_\delta \circ T_\delta \circ \Phi$, denoted by 
 $J'^\delta$, satisfies:
\begin{equation}\label{3eqstr1}
J'^\delta(0)=J_{st}.
\end{equation}
To do this we consider a linear diffeomorphism such that its differential at the origin
transforms the basis $(e_1,(T_\delta\circ \Phi)_*J'(0)(e_1),e_3,(T_\delta\circ \Phi)_*J'(0)e_3)$
into the canonical basis $(e_1,e_2,e_3,e_4)$ of $\R^4$. According to (\ref{3eqharm}) and (\ref{3eqcomp}), we have 
$$(T_\delta \circ \Phi)_*J'(0)=\Phi_*J'(\delta,0)=J'(\delta,0).$$
This means that the endomorphism   $(T_\delta \circ \Phi)_*J'(0)$ is block diagonal. 
This and the fact that $J'(\delta,0)=J'_{st}+O(\delta)$  imply that the desired diffeomorphism is expressed by:
\begin{equation}\label{3eqnorm}
\varphi_{\delta}(z):=\left(z_1+O(\delta|z_1|),z_2+O(\delta|z_2|)\right),
\end{equation}
for $z \in T_\delta(\Omega)$, and that:
\begin{equation}\label{3eqcomp2}
(J'^\delta)_\C(z)=\left(\begin{array}{cccc} 
a'_{1,\delta}(z) 
& \overline{b'_{1,\delta}(z)}  & c'_{1,\delta}(z)& \overline{c'_{2,\delta}(z)} \\
b'_{1,\delta}(z)  & \overline{a'_{1,\delta}(z)} & c'_{2,\delta}(z)& \overline{c'_{1,\delta}(z)} \\
0&0 & a'_{2,\delta}(z)  & \overline{b'_{2,\delta}(z)} \\
0 & 0 & b'_{2,\delta}(z)  & \overline{a'_{2,\delta}(z)}  \\ 
\end{array}\right), 
\end{equation}
where 
\begin{equation*}
\left\{
\begin{array}{lll}  
a'_{k,\delta}(z)&:=&a'_k(\Phi^{-1}\circ T_\delta^{-1}\circ \varphi_\delta^{-1}(z))+O(\delta) \\
\\
b'_{k,\delta}(z)&:=&b'_k(\Phi^{-1}\circ T_\delta^{-1}\circ \varphi_\delta^{-1}(z))+O(\delta)\\
\\
c'_{k,\delta}(z)&:=& c'_k(T_\delta^{-1}\circ \varphi_\delta^{-1}(z))+O(\delta)
\end{array}
\right.
\end{equation*}
for  $k=1,2$. 
Furthermore we notice that the structure $J'^\delta$ is constant and equal to 
$J_{st}+O(\delta)$ on $\R^4\setminus (\varphi_\delta\circ T_\delta \circ (\Omega))$,

%Moreover we have~: 
%\begin{equation}\label{3eqstr1}
%J^\delta(z)=J_{st}+O(\|z\|),
%\end{equation}
%for $ z\in \varphi_\delta\circ T_{\delta}(\Omega)$.

\vspace{0,5cm}

We consider now the following anisotropic dilation $\Lambda_\delta$ of $\C^2$~:
%: \varphi_\delta\circ T_{\delta}(\Omega) \rightarrow \C^2$:
$$\Lambda_\delta(z_1,z_2):=\left(\frac{z_1}{z_1+2\delta},\frac{\sqrt{2\delta}z_2}{z_1+2\delta}\right).$$ 
Its inverse is given by:
\begin{equation}\label{3eqdil}
\Lambda_\delta^{-1}(z)=\left( 2\delta\frac{z_1}{1-z_1}, 
\sqrt{2\delta}\frac{z_2}{1-z_1}\right).
\end{equation}
Let 
$$\Psi_\delta:=\Lambda_\delta \circ \varphi_\delta \circ T_\delta.$$ 
We have the following matricial representation for
the complexification of the structure $\widetilde{J^\delta}:=(\Lambda_\delta)_*J^\delta$: 
\begin{equation}\label{3eqcomp3}
\left(\begin{array}{cccc} 
A'_{1,\delta}(z) 
& \overline{B'_{1,\delta}(z)}  & C'_{1,\delta}(z)& \overline{C'_{2,\delta}(z)} \\
B'_{1,\delta}(z)  & \overline{A'_{1,\delta}((z)} & 
C'_{2,\delta}(z)& \overline{C'_{1,\delta}(z)} \\
D'_{1,\delta}(z)& \overline{D'_{2,\delta}(z)}  & A'_{2,\delta}(z)  & 
\overline{B'_{2,\delta}(z)} \\
D'_{2,\delta}(z)& \overline{D'_{1,\delta}(z)} & 
B'_{2,\delta}(z)  & \overline{A'_{2,\delta}(z)}  \\ 
\end{array}\right), 
\end{equation}
with
\begin{equation*}
\left\{
\begin{array}{lll}  
A'_{1,\delta}(z)&:=&\displaystyle a'_{1,\delta}(\Lambda_\delta^{-1}(z))+
\frac{1}{\sqrt{2\delta }}z_2c'_{1,\delta}(\Lambda_\delta^{-1}(z))\\
\\
A'_{2,\delta}(z)&:=&\displaystyle a'_{2,\delta}(\Lambda_\delta^{-1}(z))-
\frac{1}{\sqrt{2\delta }}z_2c'_{1,\delta}(\Lambda_\delta^{-1}(z))\\
\\
B'_{1,\delta}(z)&:=&\displaystyle\frac{(1-\overline{z_1})^2}{(1-z_1)^2}b'_{1,\delta}( \Lambda_\delta^{-1}(z))
+\frac{1}{\sqrt{2\delta }}\frac{(1-\overline{z_1})^2z_2}{(1-z_1)^2}c'_{2,\delta}( \Lambda_\delta^{-1}(z))\\
\\
B'_{2,\delta}(z)&:=&\displaystyle \displaystyle\frac{1-\overline{z_1}}{1-z_1} b'_{2,\delta}( \Lambda_\delta^{-1}(z)) 
-\frac{1}{\sqrt{2\delta }}\frac{(1-\overline{z_1})\overline{z_2}}{1-z_1}c'_{2,\delta}( \Lambda_\delta^{-1}(z))\\
\\
C'_{1,\delta}(z)&:=&\displaystyle \frac{1}{\sqrt{2\delta }}(1-z_1)c'_{1,\delta}(\Lambda_\delta^{-1}(z))\\
\\
C'_{2,\delta}(z)&:=& \displaystyle\frac{1}{\sqrt{2\delta }}
\frac{(1-\overline{z_1})^2}{1-z_1}c'_{2,\delta}(\Lambda_\delta^{-1}(z))\\
\\
 D'_{1,\delta}(z)&:=&\displaystyle\frac{z_2}{1-z_1}(a'_{2,\delta}( \Lambda_\delta^{-1}(z)) -a'_{1,\delta}
( \Lambda_\delta^{-1}(z)) )-\frac{1}{\sqrt{2\delta }}\frac{z_2^2}{1-z_1}c'_{1,\delta}( \Lambda_\delta^{-1}(z))\\
\\
D'_{2,\delta}(z)&:=& \displaystyle\frac{1-\overline{z_1}}{(1-z_1)^2}(z_2b'_{2,\delta}(\Lambda_\delta^{-1}(z))
-\overline{z_2}b'_{1,\delta}( \Lambda_\delta^{-1}(z)))\\
&&\\
&&-\frac{1}{\sqrt{2\delta}}\frac{(1-\overline{z_1})|z_2|^2}{(1-z_1)^2}c'_{2,\delta}( \Lambda_\delta^{-1}(z)).
\end{array}
\right.
\end{equation*}
Direct computations lead to:
\begin{equation*}
\left\{
\begin{array}{lll}  
A'_{1,\delta}(z)&=& \displaystyle a'_1(\tilde{z_1}+\rho_{2,2}\tilde{z_2}^2,\tilde{z_2})+
\frac{1}{\sqrt{2\delta }}z_2O(|\tilde{z_2}||\tilde{z_1}+\rho_{2,2}\tilde{z_2}^2|)+O(\sqrt{\delta})\\   
&&\\
B'_{1,\delta}(z)&=& \displaystyle  \frac{(1-\overline{z_1})^2}{(1-z_1)^2}b'_1(\tilde{z_1}+\rho_{2,2}\tilde{z_2}^2,\tilde{z_2})
+\frac{1}{\sqrt{2\delta }}\frac{(1-\overline{z_1})^2}{1-z_1^2}z_2
O(|\tilde{z_2}||\tilde{z_1}+\rho_{2,2}\tilde{z_2}^2|)\\
&&\\
&&+O(\sqrt{\delta})\\
&&\\
C'_{1,\delta}(z)&=&\displaystyle  \frac{1}{\sqrt{2\delta }}(1-z_1)O(|\tilde{z_2}||\tilde{z_1}+\rho_{2,2}\tilde{z_2}^2|)    
+O(\sqrt{\delta})\\
&&\\
D'_{1,\delta}(z)&=&\displaystyle \frac{z_2}{1-z_1}[(a'_2 -a'_1) ( \tilde{z_1}+\rho_{2,2}\tilde{z_2}^2,\tilde{z_2})]+
\displaystyle \frac{1}{\sqrt{2\delta }}\frac{z_2^2}{1-z_1}O(|\tilde{z_2}||\tilde{z_1}+\rho_{2,2}\tilde{z_2}^2|)\\ 
&&\\
&&+O(\sqrt{\delta}).\\
\end{array}
\right.
\end{equation*}
where 
\begin{equation*}
\left\{
\begin{array}{lll}  
\tilde{z_1}&:=& \displaystyle 2\delta\frac{z_1}{1-z_1}+\delta+O\left(\delta^2\left|\frac{z_1}{1-z_1}\right|\right)\\
&&\\
\tilde{z_2}&:=&\displaystyle \sqrt{2\delta}\frac{z_2}{1-z_1}+O\left(\delta^{3/2}\left|\frac{z_2}{1-z_1}\right|\right).
\end{array}
\right.
\end{equation*}
Notice that:
\begin{equation*}
\left\{
\begin{array}{lll}  
\displaystyle\frac{\partial}{\partial z_1}\tilde{z_1}&:=& \displaystyle 2\delta\frac{1}{(1-z_1)^2}
+\frac{\partial}{\partial z_1}O\left(\delta^2\left|\frac{z_1}{1-z_1}\right|\right)\\
&&\\
\displaystyle \frac{\partial }{\partial z_1}\tilde{z_2}&:=&\displaystyle -\sqrt{2\delta}\frac{z_2}{(1-z_1)^2}+
\frac{\partial}{\partial z_1}O\left(\delta^{3/2}\left|\frac{z_2}{1-z_1}\right|\right).
\end{array}
\right.
\end{equation*}
The crucial step is to control $\|\widetilde{J'^\delta}-J_{st}\|_{\mathcal{C}^1(\overline{\Psi_\delta(\Omega)})}$
 by some positive power of  $\delta$. Working on a small neighborhood of the unit ball $\B$ (see next Lemma \ref{3lemma}), it is sufficient to prove that 
the differential of $\widetilde{J'^\delta}$ is controlled by some positive constant of $\delta$.
We first need to determine the behaviour of a point 
$z=(z_1,z_2) \in \Psi_\delta(\Omega)$ near the infinite point $(1,0)$. Let  $\omega=(\omega_1,\omega_2) \in \Omega$ be such that $\Psi_\delta(\omega)=z$; 
then: 
\begin{eqnarray*} 
z_1=\frac{\omega_1-\delta+O(\delta|\omega_1-\delta|)}
{\omega_1+\delta+O(\delta|\omega_1-\delta|)},\\
\end{eqnarray*}
where the two terms $O(\delta|\omega_1-\delta|)$ are equal, and so 
\begin{equation}\label{3eqboule1} 
\left|\frac{1}{1-z_1}\right|= \left|\frac{\omega_1+\delta+O(\delta|\omega_1-\delta|)}{2\delta}\right|\leq c_1\delta^{-\alpha'}.
\end{equation}
for some positive constant $c_1$ independent of $z$. Moreover there is a positive constant $c_2$ such that 
\begin{equation}\label{3eqboule2} 
|z_2|=\sqrt{2\delta}\left|\frac{\omega_2+O(\delta|\omega_2|)}
{\omega_1+\delta+O(\delta|\omega_1-\delta|)}\right|\leq c_2\delta^{\alpha'/2}.\\
\end{equation}
All the behaviours being equivalent, we focus for instance on 
the derivative $\frac{\partial} {\partial z_1}D'_{1,\delta}(z)$. In this computation 
we focus only on terms that play a crucial role:
\begin{eqnarray*}
\frac{\partial} {\partial z_1}D'_{1,\delta}(z)&=&\displaystyle -\frac{z_2}{(1-z_1)^2}
[(a'_2 -a'_1)( \tilde{z_1}+\rho_{2,2}\tilde{z_2}^2,\tilde{z_2})]+\\
&&\frac{z_2}{(1-z_1)}\left[\frac{\partial} {\partial z_1}(a'_2 -a'_1).
\left(2\delta\frac{1}{(1-z_1)^2}-4\rho_{2,2}\delta\frac{z_2^2}{(1-z_1)^3}\right)\right]+\\
& & \frac{z_2}{(1-z_1)}\left[\frac{\partial} {\partial z_2}(a'_2 -a'_1).\sqrt{2\delta}\frac{z_2}{(1-z_1)^2}
\right]+\\
&&\\
&&\frac{-1}{\sqrt{2\delta }}\frac{z_2^2}{(1-z_1)^2}O(|\tilde{z_2}||\tilde{z_1}+\rho_{2,2}\tilde{z_2}^2|)\\
%&& 2\rho_{2,2}\frac{z_2^3}{(1-z_1)^2}
%\Big[\frac{\partial} {\partial z_1}(a'_2 -a'_1).
%\left(2\delta\frac{1}{(1-z_1)^2}-4\rho_{2,2}\delta\frac{z_2^2}{(1-z_1)^3}\right)\Big]+\\
%&&\\
%&& 2\rho_{2,2}\frac{z_2^3}{(1-z_1)^2}
%\left[\frac{\partial} {\partial z_2}(a'_2 -a'_1).\sqrt{2\delta}\frac{z_2}{(1-z_1)^2}
%\right]\\
&&\\
&& +
\frac{1}{\sqrt{2\delta }}\frac{z_2^2}{1-z_1}\frac{\partial}{\partial z_1}
O(|\tilde{z_2}||\tilde{z_1}+\rho_{2,2}\tilde{z_2}^2|)+R(z).\\
\end{eqnarray*}

According to (\ref{3eqboule1}), to (\ref{3eqboule2}) and to the fact that $(a_2'-a_1')(z)=O|z|$, it follows that for $\alpha'$ 
small enough 
$$\left|\frac{\partial} {\partial z_1}D'_{1,\delta}(z)\right|\leq c\delta^s$$
for  positive constants $c$ and $s$. 
By similar arguments on other derivatives, it follows  that there are positive constants, still denoted by $c$ and $s$ such that    
\begin{equation*}
\|d\widetilde{J'^\delta}\|_{\mathcal{C}^0(\overline{\Psi_\delta(\Omega)})}\leq c\delta^s.
\end{equation*}
In view of the next Lemma \ref{3lemma}, since $\Psi_\delta(\Omega)$  is bounded, this also proves that 
\begin{equation}\label{3eqstru1}
\|\widetilde{J'^\delta}-J_{st}\|_{\mathcal{C}^1(\overline{\Psi_\delta(\Omega)})}\leq c\delta^s.
\end{equation}

Moreover on $\B(0,2)\setminus \Psi_\delta(\Omega)$, by similar and easier computations we see that 
$\|\widetilde{J'^\delta}-J_{st}\|_{\mathcal{C}^1(\overline{\B(0,2)\setminus \Psi_\delta(\Omega)})}$ is also controlled by some positive constant 
of $\delta$. This finally implies the crucial control~: 
%$$\left\{
%\begin{array}{lll}
%\widetilde{J'^\delta}&:=& \widetilde{J'^\delta} \mbox{ } \mbox{ on } \mbox{ } \Psi_\delta(\Omega),\\
%\\
%\widetilde{J'^\delta}&:=& J_{st} \mbox{ } \mbox{ on } \mbox{ } \R^4\setminus (\Psi_\delta(\Omega).
%\end{array}
%\right.$$
%Moreover it is important to notice that $\widetilde{J'^\delta}$ coincides with $J_{st}$ at first order on the 
%boundary of $\Psi_\delta(\Omega)$. Finally we have 
\begin{equation}\label{3eqstru}
\left\{
\begin{array}{lll}
\widetilde{J'^\delta}(0)&=& J_{st}, \\
&&\\ 
\|\widetilde{J'^\delta}-J_{st}\|_{\mathcal{C}^1(\overline{\B(0,2)})}&\leq& c\delta^s.
\end{array}
\right.
\end{equation}

%Since $d_0\Phi=Id$, the following inclusions hold:
%\begin{equation}
%\B(0,r)\subset \Phi(\B(0,2r))\subset \B(0,4r),
%\end{equation}
%for $r$ small enough. There are two positive constants $r$ and $C$ such that: 
%\begin{enumerate}
%\item  $\overline{\B(0,3r)}\subset N_{\eps_0}$,
%\\
%\item for every $0<r'\leq 3r$,  $\partial \B(0,r')$ intersects $\partial D$ transversally,
%\\
%\item the map $\Phi$ is a diffeomorphism in $\B(0,3r)$,
%\\
%\item  $\B(0,r)\subset \Phi(\B(0,2r))\subset \B(0,4r)$, and
%\\
%\item $\overline{\Phi(D\cap \B(0,2r))}\subset \{ \|z\|^2\leq C |z_1|\}$.
%\end{enumerate}

\vspace{0,5cm}
In order to obtain estimates of the Kobayashi metric, we need to localize the domain 
$\Psi_\delta(\Omega)=\Psi_\delta(\Phi(D\cap U)\cap \Phi(Q_{(\delta,\alpha')}))$ between two balls. 
This technical result is essentially due to D.Ma \cite{ma}. 
\begin{lem}\label{3lemma}
There exists a positive constant $C$ such that:  
$$\B\left(0,e^{-C\delta^{\alpha'}}\right) \subset \Psi_\delta(\Omega) 
\subset \B\left(0,e^{C\delta^{\alpha'}}\right).$$
\end{lem} 

%\bigskip
%\begin{center}
%\input{boule.pstex_t}
%\end{center}
%\begin{center}
%Figure 2. Approximation of $\Psi_\delta(\Omega)$.
%\end{center}
%\bigskip

\begin{proof}[Proof of Lemma  \ref{3lemma}]
We have:
\begin{equation}
\Psi_\delta(z)=\left(\frac{z_1-\delta+O(\delta|z_1-\delta|)}
{z_1+\delta+O(\delta|z_1-\delta|)},\sqrt{2\delta}\frac{z_2+O(\delta|z_2|)}
{z_1+\delta+O(\delta|z_1-\delta|)}\right).
\end{equation}
Consider the following expression: 
\begin{eqnarray*}
L(z)&:=&|z_1+\delta+O(\delta|z_1-\delta|)|^2(\|\Psi_\delta(z)\|^2-1)\\
&=& |z_1-\delta+O(\delta|z_1-\delta|)|^2+
2\delta|z_2+O(\delta|z_2|)|^2\\
& & -|z_1+\delta+O(\delta|z_1-\delta|)|^2.
\end{eqnarray*}
Since $O(\delta|z_1-\delta|$) in the first and last terms of the right hand side of the previous equality  
are equal,  this leads to 
%(\ref{3eqnorm}), 
\begin{eqnarray*}
L(z)&=& 2\delta M(z)+ \delta^2O(|z_1|)+\delta^2O(|z_2|^2),
\end{eqnarray*}
where 
\begin{eqnarray*}
M(z)&:=& -2\Re e z_1+|z_2|^2.
\end{eqnarray*}
Let  $z \in \Omega =\Phi(D\cap U)\cap Q_{(\delta,\alpha')}$. For $\delta$ small enough, we have:
%Since $\Re e z_1\geq 0$ on the closure of $\Phi(\B(0,2r))$, 
\begin{eqnarray}\label{3eqre}
|z_1+\delta+O(\delta|z_1-\delta|)|^2&\geq & |z_1|^2+\delta^2+\delta^2O(|z_1|+\delta)+
\delta O(|z_1|^2+\delta|z_1|)+\nonumber\\
& &\delta^2O(|z_1|+\delta)^2+2\delta\Re e z_1\nonumber\\
& \geq & |z_1|^2+\delta^2+\delta O(|z_1|^2)+\delta^2O(|z_1|)+O(\delta^3)+2\delta\Re e z_1\nonumber\\
& \geq & \frac{3}{4}(|z_1|^2+\delta^2)+2\delta\Re e z_1.
\end{eqnarray}
Moreover 
$$2\Re e z_1 > 2\Re e \rho_{1,1} z_1^2+ 2\Re e \rho_{1,2} z_1z_2+\sum \rho_{j,\overline{k}} z_j\overline{z_k}+O(\|z\|^3).$$
Since the defining function $\rho$ is strictly $J$-plurisubharmonic, we know that, for $z$ small enough, 
$\sum \rho_{j,\overline{k}} z_j\overline{z_k}+O(\|z\|^3)$ is nonnegative.
Hence~:
$$2\Re e z_1\geq 2\Re e \rho_{1,1} z_1^2+ 2\Re e \rho_{1,2} z_1z_2$$
for $z$ sufficiently small and so there is a positive constant $C_1$ such that:
\begin{equation}\label{3eqre2}
2\Re e z_1\geq -C_1|z_1|\|z\|.
\end{equation}
Finally, (\ref{3eqre}) and (\ref{3eqre2}) lead to: 
$$|z_1+\delta+O(\delta|z_1-\delta|)|^2 \geq \frac{1}{2}(|z_1|^2+\delta^2)$$
for $z$ small enough.
Hence we have: 
\begin{equation}\label{3eqboule}
|\|\Psi_\delta(z)\|^2-1|=\frac{|L(z)|}{|z_1+\delta+O(\delta|z_1-\delta|)|^2}\leq 
\frac{4\delta |M(z)|+\delta^2O(|z_1|)+\delta^2O(|z_2|^2)}{|z_1|^2+\delta^2}.
\end{equation}

The boundary of $\Omega$ is equal to $V_1\cup V_2$ (see Figure 2), where:
$$\left\{
\begin{array}{lll}
V_1&:=&\Phi(\overline{D\cap U})\cap \partial Q_{(\delta,\alpha')},\\
\\
V_2&:=&\Phi(\partial (D\cap U))\cap Q_{(\delta,\alpha')}.\\
\end{array}
\right.$$
\bigskip
\begin{center}
\begin{picture}(0,0)%
\includegraphics{bord.pstex}%
\end{picture}%
\setlength{\unitlength}{1934sp}%
\begingroup\makeatletter\ifx\SetFigFont\undefined%
\gdef\SetFigFont#1#2#3#4#5{%
  \reset@font\fontsize{#1}{#2pt}%
  \fontfamily{#3}\fontseries{#4}\fontshape{#5}%
  \selectfont}%
\fi\endgroup%
\begin{picture}(7077,6240)(4711,-6595)
\put(7876,-4411){\makebox(0,0)[lb]{\smash{{\SetFigFont{10}{12.0}{\familydefault}{\mddefault}{\updefault}{\color[rgb]{0,0,0}$p=(\delta,0)$}%
}}}}
\put(6826,-4036){\makebox(0,0)[lb]{\smash{{\SetFigFont{10}{12.0}{\familydefault}{\mddefault}{\updefault}{\color[rgb]{0,0,0}$0$}%
}}}}
\put(10051,-5986){\makebox(0,0)[lb]{\smash{{\SetFigFont{10}{12.0}{\familydefault}{\mddefault}{\updefault}{\color[rgb]{0,0,0}$\Phi(D\cap U)$}%
}}}}
\put(6151,-6061){\makebox(0,0)[lb]{\smash{{\SetFigFont{10}{12.0}{\familydefault}{\mddefault}{\updefault}{\color[rgb]{0,0,0}$Q_{(\delta,\alpha')}$}%
}}}}
\put(4726,-886){\makebox(0,0)[lb]{\smash{{\SetFigFont{9}{10.8}{\familydefault}{\mddefault}{\updefault}{\color[rgb]{0,0,0}$V_2$}%
}}}}
\put(7351,-586){\makebox(0,0)[lb]{\smash{{\SetFigFont{9}{10.8}{\familydefault}{\mddefault}{\updefault}{\color[rgb]{0,0,0}$V_1$}%
}}}}
\end{picture}%

\end{center}
\begin{center}
Figure 2. Boundary of $\Omega$.
\end{center}
\bigskip

Let $z \in V_1$. 
%Due to (\ref{3eqpara}) and to the fact that $\B(0,r')\subset \Phi(\B(0,2r'))$, 
%there is a positive constant $C_1$ such that 
%\begin{equation}\label{3eqz}
%|z_1|\geq \frac{1}{C_1}\|z\|^2 \geq \frac{r'^2}{C_1}. 
%\end{equation}
According  (\ref{3eqboule}) we have: 
\begin{eqnarray*}
|\|\Psi_\delta(z)\|^2-1| &\leq &\frac{4\delta |M(z)|+\delta^2O(|z_1|)+\delta^2O(|z_2|^2)}{|z_1|^2+\delta^2}\\
&&\\
&\leq &\frac{4\delta |z_1|+4\delta|z_2|^2+C_2\delta^{3-\alpha'}}{\delta^{2-2\alpha'}+\delta^2}\\
&&\\
&\leq&\frac{C_3\delta^{2-\alpha'}}{\delta^{2-2\alpha'}+\delta^2} \\
&&\\
&\leq &C_4\delta^{\alpha'}\\
%&&\\
%&\leq& C_4\delta^{\frac{1-3\alpha'}{2}}. 
\end{eqnarray*}
for some positive constants $C_1$, $C_2$, $C_3$ and  $C_4$, and for $\alpha'$ small enough.

If  $z \in V_2$, then 
$$ M(z)=-2\Re e z_1+|z_2|^2=O(|z_2|^3+|z_1|\|z\|)$$
and so there is a positive constant $C_5$ such that:
\begin{equation}\label{3eqcalc}
M(z)\leq C_5 \delta^{\frac{3}{2}(1-\alpha')}.
\end{equation} 
We finally obtain from (\ref{3eqboule}) and (\ref{3eqcalc}):
\begin{eqnarray*}
|\|\Psi_\delta(z)\|^2-1|&\leq&2C_5\frac{\delta^{\frac{5-3\alpha'}{2}}}{|z_1|^2+\delta^2} 
+ C_2\frac{\delta^{3-\alpha'}}{|z_1|^2+\delta^2}\\
&&\\
&\leq& 2C_5\delta^{\frac{1-3\alpha'}{2}} +C_2\delta^{1-\alpha'}\\
&&\\
&\leq &(2C_5+C_2)\delta^{\frac{1-3\alpha'}{2}}. 
\end{eqnarray*}
This proves that:
$$\B\left(0,1-C\delta^{\alpha'}\right) 
\subset \Psi_\delta(\Omega) 
\subset \B\left(0,1+C\delta^{\alpha'}\right),$$
for some positive constant $C$.
\end{proof}

Lemma \ref{3lemma} provides for every $v\in T_0\C^2$: 
\begin{equation}\label{3eqkob1}
K_{\left(\B(0,e^{C\delta^{\alpha'}}),\widetilde{J'^\delta}\right)}(0,v)\leq K_{\left(\Psi_\delta(\Omega),\widetilde{J'^\delta}\right)}(0,v)\leq
K_{\left(\B(0,e^{-C\delta^{\alpha'}}),\widetilde{J'^\delta}\right)}(0,v).
\end{equation}

\vspace{0,5cm}

\subsection*{\bf \Large Lower estimate}

In order to give a  lower estimate of $K_{\left(\B(0,e^{C\delta^{\alpha'}}),\widetilde{J'^\delta}\right)}(0,v)$ 
we need the following proposition:
\begin{prop}\label{3thm1}
Let $\widetilde{J}$ be an almost complex structure defined on $\B\subseteq \C^2$  such that $\widetilde{J}(0)=J_{st}$. 
There exist  positive constants $\varepsilon$ and $A_\varepsilon=O(\varepsilon)$ such that if 
$\|\widetilde{J}-J_{st}\|_{\mathcal{C}^1(\B)}\leq \varepsilon$ then we have: 
\begin{equation}\label{3ee1}
K_{(\B,\widetilde{J})}(0,v)\geq \exp \left(-\frac{A_\varepsilon}{2}\right)\|v\|.
\end{equation}
\end{prop}
\begin{proof}[Proof of Proposition \ref{3thm1}]
Due to Lemma \ref{3lemchir}, there exist positive  constants $\varepsilon$ and  $A_\varepsilon=O(\varepsilon)$ such that the 
function  ${\rm log} \|z\|^2+A_\varepsilon\|z\|$ is $\widetilde{J}$-plurisubharmonic on $\B$ if 
$\|\widetilde{J}-J_{st}\|_{\mathcal{C}^{1}(\B)}\leq \varepsilon$. Consider the function $\Psi$ defined by: 
$$\Psi:= \|z\|^2e^{A_\varepsilon\|z\|}.$$

Let $u~: \Delta \rightarrow \B$ be a $\widetilde{J}$-holomorphic disc such that $u(0)=0$ and
$d_0u (\partial /\partial x)= rv$ where $v \in T_q \C^2$ and $r>0$. For $\zeta$ sufficiently close to 0 we have
$$
u(\zeta) = q + d_0u(\zeta) +
\mathcal O(|\zeta|^2).
$$
Setting $\zeta= \zeta_1+i\zeta_2$ and using
the $\widetilde{J}$-holomorphy condition $d_0u\circ J_{st} = \widetilde{J} \circ
d_0u$, we may write:
$$d_0u(\zeta) = \zeta_1 d_0u\left(\frac{\partial}{\partial x}\right) + 
\zeta_2 \widetilde{J}\left(d_0u\left(\frac{\partial}{\partial x}\right)\right).$$
This implies
\begin{equation}\label{3eqcalc2}
|d_0u(\zeta)| \leq |\zeta| \|I+\widetilde{J}\| \left\|d_0u\left(\frac{\partial}{\partial x}\right)\right\|.
\end{equation}
 
We now consider the following function
$$\phi(\zeta):= \frac{\Psi(u(\zeta))}{|\zeta|^2}=\frac{\|u(\zeta)\|^2}{|\zeta|^2} \exp(A_\varepsilon|u(\zeta)|) ,$$
which is subharmonic on $\Delta \backslash \{0\}$ since $\log \phi$ is subharmonic. 
According to (\ref{3eqcalc2}) \\ 
$\limsup_{\zeta \rightarrow 0}\phi(\zeta) $ is finite. 
Moreover setting $\zeta_2=0$ we have:
$$ \limsup_{\zeta \rightarrow 0}\phi(\zeta) \geq \left\|d_0u\left(\frac{\partial}{\partial x}\right)\right\|^2.
$$
Applying the maximum principle to a subharmonic extension of $\phi$
on $\Delta$ we obtain the inequality:
$$
\left\|d_0u\left(\frac{\partial}{ \partial x}\right)\right\|^2 \leq \exp A_\varepsilon.
$$

Hence, by definition of the Kobayashi infinitesimal metric, 
we obtain for every $q \in D \cap V$, $v \in T_q M$:
\begin{equation}\label{3localhyp}
K_{(D,\widetilde{J})}(q,v) \geq  \exp\left(-\frac{A_\varepsilon}{2}\right)\|v\|.
\end{equation}

This gives the desired estimate (\ref{3ee1}).
\end{proof}

In order to  apply Proposition \ref{3thm1} to the structure $\widetilde{J'^\delta}$, it is necessary to dilate 
isotropically the ball $\B(0,e^{C\delta^{\alpha'}})$ to the unit ball $\B$. So consider the dilation of $\C^2$: 
$$\Gamma(z)=e^{-C\delta^{\alpha'}}z.$$
\begin{equation}\label{3eqkob3}
K_{\left(\B(0,e^{C\delta^{\alpha'}}),\widetilde{J'^\delta}\right)}(0,v)=
e^{-C\delta^{\alpha'}}K_{\left(\B,\Gamma_*\widetilde{J'^\delta}\right)}(0,v).
\end{equation}
According to (\ref{3eqkob1}) we obtain: 
\begin{equation}\label{3eqkob2}
e^{-C\delta^{\alpha'}}K_{\left(\B,\Gamma_*\widetilde{J'^\delta}\right)}(0,v)\leq 
K_{\left(\Psi_\delta(\Omega),\widetilde{J'^\delta}\right)}(0,v).
\end{equation}
Then applying Proposition \ref{3thm1} to the structure $\Gamma_*\widetilde{J'^\delta}=\widetilde{J'^\delta}(e^{C\delta^{\alpha'}}.)$ 
and to $\varepsilon=c\delta^s$ (see (\ref{3eqstru})) provides the existence of a positive constant $C_1$ such that: 
\begin{equation}\label{3ee2}
K_{\left(\B,\Gamma_*\widetilde{J'^\delta}\right)}(0,v)\geq e^{-C_1\delta^s}\|v\|.
\end{equation} 
Moreover 
\begin{equation}\label{3eqkob4}
K_{\left(\Omega,J' \right)}((\delta,0),v)=
K_{\left(\Psi_\delta(\Omega),\widetilde{J'^\delta}\right)}(0,d_{(\delta,0)}\Psi_\delta (v)),
\end{equation}
where 
\begin{eqnarray*}
\displaystyle d_{(\delta,0)}\Psi_\delta (v)&=&\displaystyle d_0\Lambda_\delta \circ d_0\varphi_\delta \circ d_{(\delta,0)}T_\delta(v)\\
&&\\
&=& \displaystyle \left(\frac{1}{2\delta}(v_1+O(\delta)v_1),\frac{1}{\sqrt{2\delta}}(v_2+O(\delta)v_2)\right).
\end{eqnarray*}
According to (\ref{3eqloc3}), (\ref{3ee2}), (\ref{3eqkob2}) and (\ref{3eqkob4}), we finally obtain:  
\begin{equation}\label{3eqfinalement}
K_{\left(D,J\right)}(p,v)\geq  e^{-C_2\delta^{\beta''}}\left(\frac{|v_1|^2}{4\delta^2}+
\frac{|v_2|^2}{2\delta}\right)^{\frac{1}{2}},
\end{equation}
for some positive constant $C_2$ and $\beta''$.

\vspace{0.5cm}

\subsection*{\bf \Large Upper estimate}

Now, we want to prove the existence of a positive constant $C_3$ such that
\begin{equation*}
K_{\left(D,J\right)}(p,v)\leq  e^{C_3\delta^{\alpha'}}\left(\frac{|v_1|^2}{4\delta^2}+
\frac{|v_2|^2}{2\delta}\right)^{\frac{1}{2}}.\\
\end{equation*} 
According  to the decreasing  property of the Kobayashi metric it is sufficient to give an upper estimate for 
$K_{(\Phi(D\cap U)\cap Q_{(\delta,\alpha)},J)}(p,v)$.  
%\begin{equation}\label{3eqloca1}
%K_{(D,J)}(p,v)\leq K_{(D\cap Q_{(\delta,\alpha)},J)}(p,v)
%\end{equation}
%for every $v\in T_p \R^4$. Thus we have to estimate $K_{(D\cap Q_{(\delta,\alpha)},J)}(p,v)$.  
%The fondamental difference with the lower estimate is that 
%we do not need  to extend  the structure $J$ by $J_{st}$ away from a neighborhood of the orgin. Thus we introduce as 
%previously  diffeomorphisms $\Phi$, $\Psi_\delta:=\Lambda_\delta \circ \varphi_\delta \circ T_\delta$ and the set 
%$\Omega= Q_{(\delta,\alpha)}$.  The direct image of the structure $J$ by $\Psi_\delta \circ \Phi$ is denoted by 
%$\widetilde{J^\delta}$ and still satisfies (\ref{3eqstru1}). Moreover,  Lemma  \ref{3lemma} still holds for 
%$\beta'=\frac{1-3\alpha}{2}$.
Moreover, due to  (\ref{3eqkob1}) and (\ref{3eqkob4}) it is sufficient to prove:
\begin{equation}\label{3ee4}
K_{\left(\B(0,e^{-C\delta^{\alpha'}}),\widetilde{J^\delta}\right)}(0,v)\leq e^{C_4\delta^{\alpha'}}\|v\|.
\end{equation} 
%We consider the following  $J_{st}$-holomorphic disc $u~: \Delta \rightarrow \B(0,e^{-C\delta^{\alpha'}})$, defined by 
%$u(\zeta)=\zeta v/\|v\|$.
In that purpose we need to deform quantitatively a standard holomorphic disc  contained in the ball 
$\B(0,e^{-C\delta^{\alpha'}})$ into a 
$\widetilde{J^\delta}$-holomorphic disc, controlling the size of the new disc, and 
consequently its derivative at the origin. As previously by dilating isotropically the 
ball  $\B(0,e^{-C\delta^{\alpha'}})$ into the unit ball $\B$, 
 we may suppose that we work on the unit ball endowed with  $\widetilde{J^\delta}$ satisfying (\ref{3eqstru}).

We define for a  map  $g$ with values in a complex
vector space, continuous on $\overline \Delta$, and for $z\in \Delta$ the {\it Cauchy-Green operator} by:
$$T_{CG}(g)(z):=\frac{1}{\pi} \int_\Delta\frac{g(\zeta)}{z-\zeta}dxdy.$$
%We need the following  properties of $T_{CG}$:
%\begin{enumerate}
%\item If $g\in \mathcal{C}^{k,\alpha}(\overline \Delta ),~k\in \N~,0<\alpha<1~,$
%then $T_{CG}g\in \mathcal{C}^{k+1,\alpha}(\overline \Delta )$.
%\item ${\partial \over \partial\overline z}[T_{CG}(g)]=g$. (on $\overline \Delta$.)
%\end{enumerate}
We consider now the operator $\Phi_{\widetilde{J^\delta}}$ from
$\mathcal{C}^{1,r}(\overline \Delta,\B(0,2))$ into $\mathcal{C}^{1,r}(\overline \Delta,\R^4)$ by:
$$
\Phi_{\widetilde{J^\delta}}(u):=\left(Id-T_{CG}q_{\widetilde{J^\delta}}(u)\frac{\partial}{ \partial z}\right)u,
$$
which is well defined  since $\widetilde{J^\delta}$ satisfying (\ref{3eqstru}).
Let  $u~: \Delta \rightarrow \B$  be a $\widetilde{J^\delta}$-holomorphic disc in $\mathcal{C}^{1,r}(\overline \Delta,\B)$. 
According to the continuity of the Cauchy-Green operator from $\mathcal{C}^r(\overline \Delta,\R^4)$ into 
$\mathcal{C}^{1,r}(\overline \Delta,\R^4)$ and since $\widetilde{J^\delta}$ satisfies (\ref{3eqstru}), we get: 
\begin{eqnarray*}
\left\|T_{CG}q_{\widetilde{J^\delta}}(u)\frac{\partial}{ \partial z}u\right\|_{\mathcal{C}^{1,r}(\overline{\Delta})}&\leq& 
c\left\|q_{\widetilde{J^\delta}}(u)\frac{\partial}{ \partial z}u\right\|_{\mathcal{C}^{r}(\overline{\Delta})} \\
&&\\
&\leq &c\left\|q_{\widetilde{J^\delta}}\right\|_{\mathcal{C}^1(\overline{\B})}\|u\|_{\mathcal{C}^{1,r}(\overline{\Delta})}\\
&&\\
&\leq &c'\left\|\widetilde{J^\delta}-J_{st}\right\|_{\mathcal{C}^1(\overline{\B})}\|u\|_{\mathcal{C}^{1,r}(\overline{\Delta})}\\
&&\\
&\leq &c''\delta^s\|u\|_{\mathcal{C}^{1,r}(\overline{\Delta})}
\end{eqnarray*}
for some positive constants $c$, $c'$ and $c''$. 
Hence 
\begin{equation}\label{3eqphidef}
(1-c''\delta^{s})\|u\|_{\mathcal{C}^{1,r}(\overline{\Delta})} \leq 
\left\|\Phi_{\widetilde{J^\delta}}(u)\right\|_{\mathcal{C}^{1,r}(\overline{\Delta})}
\leq (1+c''\delta^{s})\|u\|_{\mathcal{C}^{1,r}(\overline{\Delta})} 
\end{equation}
for any $\widetilde{J^\delta}$-holomorphic disc $u~: \Delta \rightarrow \B$. 
This implies that the map $\Phi_{\widetilde{J^\delta}}$ is a $\mathcal{C}^1$ diffeomorphism  
 from $\mathcal{C}^{1,r}(\overline{\Delta},\B)$ onto $\Phi_{\widetilde{J^\delta}}(\mathcal{C}^{1,r}(\overline{\Delta},\B))$.
Furthermore the following property is classical: the disc   $u$ is $\widetilde{J^\delta}$-holomorphic if and only if 
$\Phi_{\widetilde{J^\delta}}(u)$ is $J_{st}$-holomorphic. According to (\ref{3eqphidef}),
there exists a positive constant $c_3$ such that for $w \in \R^{4}$ with $\|w\|= 1-c_3\delta^{s}$, the 
map $h_w~: \Delta \rightarrow \B(0,1-c_3\delta^s)$  defined by  $h_w(\zeta)=\zeta w$ belongs to 
$\Phi_{\widetilde{J^\delta}}(\mathcal{C}^{1,r}(\overline{\Delta},\B))$. In particular, 
the map $\Phi_{\widetilde{J^\delta}}^{-1}(h_w)$ is a $\widetilde{J^\delta}$-holomorphic disc from $\Delta$ 
to the unit ball $\B$.

% $\Phi_{\widetilde{J^\delta}}$ is a  continuous differentiable diffeomorphism from 
%from $\mathcal{C}^{1,r}(\overline \Delta,\B)$ into 
%itself, whose 
%derivative at the point $u\in \mathcal{C}^{1,r}(\overline \Delta,\R^4)$
%is the map:
%\begin{equation}\label{3eqdeform1}
%d_u\Phi_{\widetilde{J^\delta}}(v')=
%\Phi_{\widetilde{J^\delta}}(v')-T_{CG}S(v') \frac{\partial}{\partial z}u,
%\end{equation}
%where $S$ is the operator defined by differentiation of
%$q_{\widetilde{J^\delta}}(u)$ defined by:
%\begin{eqnarray}\label{3eqdeform2}
%S(v') &= &\left(\widetilde{J^\delta}(u)+J_{st}\right)^{-1} d\widetilde{J^\delta}(v')\nonumber\\
%& & ~-~
%\left(\widetilde{J^\delta}(u)+J_{st}\right)^{-1} d\widetilde{J^\delta}(v') \left(\widetilde{J^\delta}(u)+
%J_{st}\right)^{-1}\left(\widetilde{J^\delta}(u)-J_{st}\right).
%\end{eqnarray}
%According  to (\ref{3eqstru}) we have ~:
%\begin{equation*}
%\|S(v')\| \leq c_1\delta^{s_1}\|v'\|,
%\end{equation*}
%for some positive constants $c_1$ and $s_1$. Hence, , we have:  
%\begin{equation}\label{3eqdeformav}
%(1-c_2\delta^{s_2})\|v'\|\leq \|d_u\Phi_{\widetilde{J^\delta}}(v')\|
%\leq (1+c_2\delta^{s_2})\|v'\|
%\end{equation}
%for some positive constants $c_2$ and $s_2$. 

%fficiently small $\delta$, the map $\Phi_{\widetilde{J^\delta}}^{-1}$ is defined
%  from $\mathcal{C}^{1,r}(\overline \Delta,\R^4)$ 
%into the unit ball of $\mathcal{C}^{1,r}(\overline \Delta,\R^4)$. 

%Let $h$ be a $J_{st}$-holomorphic disc such that  $u:=\Phi_{\widetilde{J^\delta}}^{-1}(h)$ is well defined.  

Consider now $w \in \R^{4}$ such that $\|w\|= 1-c_3\delta^{s}$, and  $h_w$ the associated standard holomorphic disc. 
Let us estimate the derivative of the $\widetilde{J^\delta}$-holomorphic disc 
$u:=\Phi_{\widetilde{J^\delta}}^{-1}(h_w)$  at the origin:   
\begin{eqnarray}\label{3eqcald}
w&=& \frac{\partial h}{\partial x}(0)\nonumber \\
&=& \frac{\partial}{\partial x}\left(\Phi_{\widetilde{J^\delta}}(u)\right)(0) \nonumber\\
&&\nonumber\\
&=& \frac{\partial}{\partial x}u(0)+ \frac{\partial}{\partial x} T_{CG}q_{\widetilde{J^\delta}}(u)\frac{\partial u}{ \partial z}\nonumber\\
&&\nonumber\\
&=& \frac{\partial}{\partial x}u(0)+ T_{CZ}\left(q_{\widetilde{J^\delta}}(u)\frac{\partial u}{ \partial z}\right)(0)
\end{eqnarray}
where $T_{CZ}$ denotes the {\it Calderon-Zygmund} operator. This is defined by: 
$$T_{CZ}(g)(z):=\frac{1}{\pi} \int_\Delta\frac{g(\zeta)}{(z-\zeta)^2}dxdy,$$
for a  map  $g$ with values in a complex vector space, continuous on $\overline \Delta$ and  for $z\in \Delta$, with the integral in the sense of principal value. 
Since  $T_{CZ}$ is a continuous operator  from $\mathcal{C}^r(\overline{\Delta},\R^4)$ into $\mathcal{C}^r(\overline{\Delta},\R^4)$, 
we have:
\begin{equation}\label{3eqcald2}
\left\|T_{CZ}\left(q_{\widetilde{J^\delta}}(u)\frac{\partial u}{ \partial z}\right)(0)\right\| \leq  
c\left\|q_{\widetilde{J^\delta}}(u)\frac{\partial}{ \partial z}u\right\|_{\mathcal{C}^{r}(\overline{\Delta})} 
\leq c'''\delta^s\|u\|_{\mathcal{C}^{1,r}(\overline{\Delta})}
\end{equation}
for some positive constant $c$ and $c'''$. Moreover, according to (\ref{3eqphidef}) we have: 
\begin{equation}\label{3eqphidef2}
\left\|u\right\|_{\mathcal{C}^{1,r}(\overline{\Delta})}=
\left\|\Phi_{\widetilde{J^\delta}}^{-1}(h_w)\right\|_{\mathcal{C}^{1,r}(\overline{\Delta})}
\leq (1+c''\delta^{s})\|h_w\|_{\mathcal{C}^{1,r}(\overline{\Delta})}\leq 2 \|w\|. 
\end{equation}
Finally (\ref{3eqcald}), (\ref{3eqcald2}) and (\ref{3eqphidef2}) lead to: 
\begin{equation}\label{3eqcald3}
(1-2c'''\delta^s)\|w\|\leq  \left\|\frac{\partial}{\partial x}\left(\Phi_{\widetilde{J^\delta}}^{-1}(h_w)\right)(0)\right\|  \leq  
(1+ 2c'''\delta^s)\|w\|.
\end{equation}
This implies that the map $ \displaystyle w \mapsto \frac{\partial}{\partial x}\left(\Phi_{\widetilde{J^\delta}}^{-1}h_{w}\right)(0)$ 
is a small continuously differentiable perturbation of the identity. More precisely, using (\ref{3eqcald3}),
 there exists a positive constant $c_4$ such that for every vector $v \in \R^4\setminus \{0\}$ and for 
$r=1-c_4\delta^{s}$, there is a vector $w\in \R^4$ satisfying  $\|w\|\leq 1+c_3\delta^{s}$ and such that 
$\frac{\partial}{\partial x}\left(\Phi_{\widetilde{J^\delta}}^{-1}h_{w}\right)(0)=rv/\|v\|$ (see Figure 3). 
\bigskip
\begin{center}
\begin{picture}(0,0)%
\includegraphics{deform.pstex}%
\end{picture}%
\setlength{\unitlength}{2250sp}%
\begingroup\makeatletter\ifx\SetFigFont\undefined%
\gdef\SetFigFont#1#2#3#4#5{%
  \reset@font\fontsize{#1}{#2pt}%
  \fontfamily{#3}\fontseries{#4}\fontshape{#5}%
  \selectfont}%
\fi\endgroup%
\begin{picture}(6405,5177)(2236,-6444)
\put(6076,-4186){\makebox(0,0)[lb]{\smash{{\SetFigFont{11}{13.2}{\familydefault}{\mddefault}{\updefault}{\color[rgb]{0,0,0}$0$}%
}}}}
\put(2251,-3811){\makebox(0,0)[lb]{\smash{{\SetFigFont{11}{13.2}{\familydefault}{\mddefault}{\updefault}{\color[rgb]{0,0,0}$\Phi_{\widetilde{J^\delta}}^{-1}h_{w}$}%
}}}}
\put(3826,-6286){\makebox(0,0)[lb]{\smash{{\SetFigFont{11}{13.2}{\familydefault}{\mddefault}{\updefault}{\color[rgb]{0,0,0}$h_w$}%
}}}}
\put(8626,-3436){\makebox(0,0)[lb]{\smash{{\SetFigFont{11}{13.2}{\familydefault}{\mddefault}{\updefault}{\color[rgb]{0,0,0}$w$}%
}}}}
\put(6826,-1486){\makebox(0,0)[lb]{\smash{{\SetFigFont{11}{13.2}{\familydefault}{\mddefault}{\updefault}{\color[rgb]{0,0,0}$rv/\|v\|$}%
}}}}
\end{picture}%

\end{center}
\begin{center}
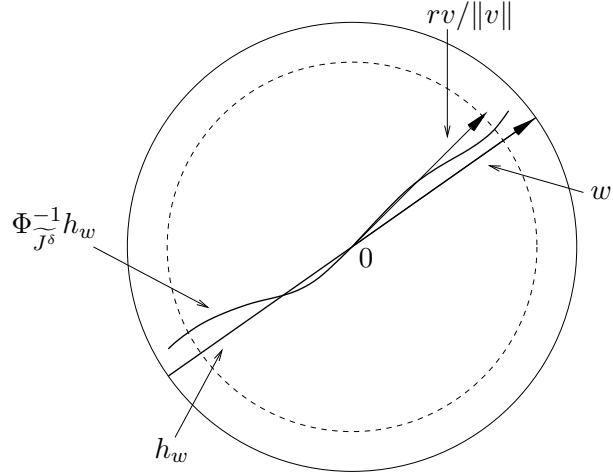
Figure 3. Deformation of a standard holomorphic disc.
\end{center}
\bigskip
Hence the  $\widetilde{J^\delta}$-holomorphic disc $\Phi_{\widetilde{J^\delta}}^{-1}h_{w}~: \Delta \rightarrow \B$ satisfies 
\begin{equation*}
\left\{
\begin{array}{lll}  
\Phi_{\widetilde{J^\delta}}^{-1}h_{w}(0)& = &0 ,\\
&&\\
\frac{\partial}{\partial x}\Phi_{\widetilde{J^\delta}}^{-1}h_{w}(0)&=&r\frac{v}{\|v\|}.\\
\end{array}
\right.
\end{equation*}
This proves estimate (\ref{3ee4}), giving the upper estimate  of Theorem A. 
 
The lower estimate (\ref{3eqfinalement}) and the upper estimate (\ref{3ee4}) imply estimate (\ref{3est2a}) of  Theorem A.

\end{proof}

\end{document}